\documentclass[reqno]{amsart}

\usepackage{mathabx} 
\usepackage{graphicx}
\usepackage{amssymb,amsmath,amsthm,amscd}
\usepackage{mathrsfs}
\usepackage{enumerate}
\usepackage{enumitem}
\usepackage{verbatim}
\usepackage[usenames,dvipsnames]{color}
\usepackage[colorlinks=true, pdfstartview=FitV,
linkcolor=blue,citecolor=blue,urlcolor=blue]{hyperref}
\usepackage[all]{xy}
\usepackage{tikz}
\usepackage{tikz-cd}
\usetikzlibrary{matrix, arrows}
\usepackage{stmaryrd} 
\usepackage{dsfont}
\usepackage{bm} 
\usepackage{array}

\usepackage{mathrsfs}

\allowdisplaybreaks[4]

\allowdisplaybreaks[1]

\setlist[itemize,1]{leftmargin=.4in}
\setlist[enumerate,1]{leftmargin=.4in,label=(\alph*)}
\setlist[description,1]{leftmargin=.4in,font=\normalfont\itshape}

\DeclareSymbolFont{sansops}{OT1}{\sfdefault}{m}{n}
\SetSymbolFont{sansops}{bold}{OT1}{\sfdefault}{b}{n}

\makeatletter
\renewcommand\operator@font{\mathgroup\symsansops}
\makeatother

\newcommand{\nc}{\newcommand}
\nc{\rnc}{\renewcommand}
\nc{\on}{\operatorname}

\numberwithin{equation}{section}

\usepackage[colorinlistoftodos]{todonotes}
\usepackage{mathtools}
\mathtoolsset{showmanualtags}
\mathtoolsset{showonlyrefs}
\usepackage{empheq}

\theoremstyle{plain}
\newtheorem{lemma}{Lemma}[subsection]
\newtheorem{prop}[lemma]{Proposition}
\newtheorem{theorem}[lemma]{Theorem}
\newtheorem{theoremintro}{Theorem}
\nc{\Prop}{\begin{prop}}
	\nc{\enprop}{\end{prop}}
\nc{\Lemma}{\begin{lemma}}
	\nc{\enlemma}{\end{lemma}}
\nc{\Th}{\begin{theorem}}
	\nc{\enth}{\end{theorem}}
\newtheorem{corollary}[lemma]{Corollary}
\nc{\Cor}{\begin{corollary}}
	\nc{\encor}{\end{corollary}}
\newtheorem{definition}[lemma]{Definition}

\nc{\Def}{\begin{definition}}
	\nc{\edf}{\end{definition}}
\newtheorem{sublemma}[lemma]{Sublemma}
\nc{\Sublemma}{\begin{sublemma}}
	\nc{\ensub}{\end{sublemma}}

\theoremstyle{definition}
\newtheorem{remark}[lemma]{Remark}
\newtheorem{remarks}[lemma]{Remarks}
\newtheorem{example}[lemma]{Example}
\newtheorem{Convention}[lemma]{Convention}
\nc{\Conv}{\begin{Convention}}
	\nc{\enconv}{\end{Convention}}
\nc{\Rem}{\begin{remark}}
	\nc{\enrem}{\end{remark}}

\nc{\rmkend}{\hfill$\triangledown$}
\nc{\defend}{\hfill$\triangle$}

\nc{\be}{\begin{enumerate}}
	\nc{\ee}{\end{enumerate}}
\nc{\eq}{\begin{eqnarray}}
	\nc{\eneq}{\end{eqnarray}}
\nc{\bc}{\begin{cases}}
	\nc{\ec}{\end{cases}}
\nc{\eqn}{\begin{eqnarray*}}
	\nc{\eneqn}{\end{eqnarray*}}
\nc{\ba}{\begin{array}}
	\nc{\ea}{\end{array}}
	
\rnc{\eq}[1]{\begin{equation} #1 \end{equation}}

\nc{\arxiv}[1]{\href{http://arxiv.org/abs/#1}{\tt arXiv:\nolinkurl{#1}}}

\nc{\ourcomment}[1]{}
\nc{\andreacomment}[1]{{\color{magenta}{//Andrea: #1//}}}
\nc{\bartcomment}[1]{{\color{blue}{//Bart: #1//}}}
\nc{\Omit}[1]{}
\nc{\summary}[1]{}
\nc{\tm}[1]{{\color{magenta}{#1}}}

\nc {\ie}{{\emph{i.e.}}, } 
\nc {\eg}{{\emph{e.g.}}, } 
\nc {\cf}{{\emph{cf.}} } 
\nc {\loccit}{{\emph{loc. cit. }} } 

\nc {\ul}{\underline}
\nc {\ol}{\overline}
\nc {\wtil}{\widetilde}
\nc {\wh}{\widehat}
\nc {\wb}{\widebar}
\nc{\scs}{\scriptscriptstyle}
\nc{\scsop}{\scriptscriptstyle\operatorname} 
\nc {\aand}{\qquad\mbox{and}\qquad}
\nc{\qu}{\quad}
\nc{\qq}{\qquad}

\nc{\drc}[1]{\delta_{#1}}
\nc{\der}{\partial}
\nc{\sfad}{\mathsf{ad}}
\nc{\ten}{\otimes}

\nc{\cA}{{\mathcal A}}
\nc{\cB}{{\mathcal B}}
\nc{\cC}{{\mathcal C}}
\nc{\cD}{{\mathcal D}}
\nc{\cE}{{\mathcal E}}
\nc{\cF}{{\mathcal F}}
\nc{\cG}{{\mathcal G}}
\nc{\cH}{{\mathcal H}}
\nc{\cI}{{\mathcal I}}
\nc{\cJ}{{\mathcal J}}
\nc{\cK}{{\mathcal K}}
\nc{\cL}{{\mathcal L}}
\nc{\cM}{\mathcal{M}}
\nc{\cN}{{\mathcal N}}
\nc{\cO}{{\mathcal O}}
\nc{\cP}{{\mathcal P}}
\nc{\calQ}{{\mathcal Q}}
\nc{\cR}{{\mathcal R}}
\nc{\cS}{\mathcal{S}}
\nc{\cT}{{\mathcal T}}
\nc{\cU}{\mathcal{U}}
\nc{\cV}{{\mathcal V}}
\nc{\cX}{{\mathcal X}}
\nc{\cY}{\mathcal{Y}}
\nc{\cW}{\mathcal{W}}
\nc{\cZ}{{\mathcal Z}}

\nc{\bbA}{{\mathbb{A}}}
\nc{\bbB}{{\mathbb{B}}}
\nc{\bbC}{{\mathbb{C}}}
\nc{\bbD}{{\mathbb{D}}}
\nc{\bbE}{{\mathbb{E}}}
\nc{\bbF}{{\mathbb{F}}}
\nc{\bbG}{{\mathbb{G}}}
\nc{\bbH}{{\mathbb{H}}}
\nc{\bbI}{{\mathbb{I}}}
\nc{\bbJ}{{\mathbb{J}}}
\nc{\bbK}{{\mathbb{K}}}
\nc{\bbL}{{\mathbb{L}}}
\nc{\bbM}{{\mathbb{M}}}
\nc{\bbN}{{\mathbb{N}}}
\nc{\bbO}{{\mathbb{O}}}
\nc{\bbP}{{\mathbb{P}}}
\nc{\bbQ}{{\mathbb{Q}}}
\nc{\bbR}{{\mathbb{R}}}
\nc{\bbS}{{\mathbb{S}}}
\nc{\bbT}{{\mathbb{T}}}
\nc{\bbU}{{\mathbb{U}}}
\nc{\bbV}{{\mathbb{V}}}
\nc{\bbX}{{\mathbb{X}}}
\nc{\bbY}{{\mathbb{Y}}}
\nc{\bbW}{{\mathbb{W}}}
\nc{\bbZ}{{\mathbb{Z}}}

\nc{\scrA}{{\mathscr A}}
\nc{\scrB}{{\mathscr B}}
\nc{\scrC}{{\mathscr C}}
\nc{\scrD}{{\mathscr D}}
\nc{\scrE}{{\mathscr E}}
\nc{\scrF}{{\mathscr F}}
\nc{\scrG}{{\mathscr G}}
\nc{\scrH}{{\mathscr H}}
\nc{\scrI}{{\mathscr I}}
\nc{\scrJ}{{\mathscr J}}
\nc{\scrK}{{\mathscr K}}
\nc{\scrL}{{\mathscr L}}
\nc{\scrM}{{\mathscr M}}
\nc{\scrN}{{\mathscr N}}
\nc{\scrO}{{\mathscr O}}
\nc{\scrP}{{\mathscr P}}
\nc{\scrQ}{{\mathscr Q}}
\nc{\scrR}{{\mathscr R}}
\nc{\scrS}{{\mathscr S}}
\nc{\scrT}{{\mathscr T}}
\nc{\scrU}{{\mathscr U}}
\nc{\scrV}{{\mathscr V}}
\nc{\scrX}{{\mathscr X}}
\nc{\scrY}{{\mathscr Y}}
\nc{\scrW}{{\mathscr W}}
\nc{\scrZ}{{\mathscr Z}}

\nc{\sfA}{{\mathsf A}}
\nc{\sfB}{{\mathsf B}}
\nc{\sfC}{{\mathsf C}}
\nc{\sfD}{{\mathsf D}}
\nc{\sfE}{{\mathsf E}}
\nc{\sfF}{{\mathsf F}}
\nc{\sfG}{{\mathsf G}}
\nc{\sfH}{{\mathsf H}}
\nc{\sfI}{{\mathsf I}}
\nc{\sfJ}{{\mathsf J}}
\nc{\sfK}{{\mathsf K}}
\nc{\sfL}{{\mathsf L}}
\nc{\sfM}{{\mathsf M}}
\nc{\sfN}{{\mathsf N}}
\nc{\sfO}{{\mathsf O}}
\nc{\sfP}{{\mathsf P}}
\nc{\sfQ}{{\mathsf Q}}
\nc{\sfR}{{\mathsf R}}
\nc{\sfS}{{\mathsf S}}
\nc{\sfT}{{\mathsf T}}
\nc{\sfU}{{\mathsf U}}
\nc{\sfV}{{\mathsf V}}
\nc{\sfX}{{\mathsf X}}
\nc{\sfY}{{\mathsf Y}}
\nc{\sfW}{{\mathsf W}}
\nc{\sfZ}{{\mathsf Z}}

\nc{\sfa}{{\mathsf a}}
\nc{\sfb}{{\mathsf b}}
\nc{\sfc}{{\mathsf c}}
\nc{\sfd}{{\mathsf d}}
\nc{\sfe}{{\mathsf e}}
\nc{\sff}{{\mathsf f}}
\nc{\sfg}{{\mathsf g}}
\nc{\sfh}{{\mathsf h}}
\nc{\sfi}{{\mathsf i}}
\nc{\sfj}{{\mathsf j}}
\nc{\sfk}{{\mathsf k}}
\nc{\sfl}{{\mathsf l}}
\nc{\sfm}{{\mathsf m}}
\nc{\sfn}{{\mathsf n}}
\nc{\sfo}{{\mathsf o}}
\nc{\sfp}{{\mathsf p}}
\nc{\sfq}{{\mathsf q}}
\nc{\sfr}{{\mathsf r}}
\nc{\sfs}{{\mathsf s}}
\nc{\sft}{{\mathsf t}}
\nc{\sfu}{{\mathsf u}}
\nc{\sfv}{{\mathsf v}}
\nc{\sfx}{{\mathsf x}}
\nc{\sfy}{{\mathsf y}}
\nc{\sfw}{{\mathsf w}}
\nc{\sfz}{{\mathsf z}}

\nc {\bfA}{{\mathbf A}}
\nc {\bfB}{{\mathbf B}}
\nc {\bfC}{{\mathbf C}}
\nc {\bfD}{{\mathbf D}}
\nc {\bfE}{{\mathbf E}}
\nc {\bfF}{{\mathbf F}}
\nc {\bfG}{{\mathbf G}}
\nc {\bfH}{{\mathbf H}}
\nc {\bfI}{{\mathbf I}}
\nc {\bfJ}{{\mathbf J}}
\nc {\bfK}{{\mathbf K}}
\nc {\bfL}{{\mathbf L}}
\nc {\bfM}{{\mathbf M}}
\nc {\bfN}{{\mathbf N}}
\nc{\bfO}{{\mathbf O}}
\nc {\bfP}{{\mathbf P}}
\nc {\bfQ}{{\mathbf Q}}
\nc {\bfR}{{\mathbf R}}
\nc {\bfS}{{\mathbf S}}
\nc {\bfT}{{\mathbf T}}
\nc {\bfU}{{\mathbf U}}
\nc {\bfV}{{\mathbf V}}
\nc {\bfX}{{\mathbf X}}
\nc {\bfY}{{\mathbf Y}}
\nc {\bfW}{{\mathbf W}}
\nc {\bfZ}{{\mathbf Z}}

\nc {\fka}{{\mathfrak a}}
\nc {\fkb}{{\mathfrak b}}
\nc {\fkc}{{\mathfrak c}}
\nc {\fkd}{{\mathfrak d}}
\nc {\fke}{{\mathfrak e}}
\nc {\fkf}{{\mathfrak f}}
\nc {\fkg}{{\mathfrak g}}
\nc {\fkh}{{\mathfrak h}}
\nc {\fki}{{\mathfrak i}}
\nc {\fkj}{{\mathfrak j}}
\nc {\fkk}{{\mathfrak k}}
\nc {\fkl}{{\mathfrak l}}
\nc {\fkm}{{\mathfrak m}}
\nc {\fkn}{{\mathfrak n}}
\nc {\fko}{{\mathfrak o}}
\nc {\fkp}{{\mathfrak p}}
\nc {\fkq}{{\mathfrak q}}
\nc {\fkr}{{\mathfrak r}}
\nc {\fks}{{\mathfrak s}}
\nc {\fkt}{{\mathfrak t}}
\nc {\fku}{{\mathfrak u}}
\nc {\fkv}{{\mathfrak v}}
\nc {\fkx}{{\mathfrak x}}
\nc {\fky}{{\mathfrak y}}
\nc {\fkw}{{\mathfrak w}}
\nc {\fkz}{{\mathfrak z}}

\nc {\fkA}{{\mathfrak A}}
\nc {\fkB}{{\mathfrak B}}
\nc {\fkC}{{\mathfrak C}}
\nc {\fkD}{{\mathfrak D}}
\nc {\fkE}{{\mathfrak E}}
\nc {\fkF}{{\mathfrak F}}
\nc {\fkG}{{\mathfrak G}}
\nc {\fkH}{{\mathfrak H}}
\nc {\fkI}{{\mathfrak I}}
\nc {\fkJ}{{\mathfrak J}}
\nc {\fkK}{{\mathfrak K}}
\nc {\fkL}{{\mathfrak L}}
\nc {\fkM}{{\mathfrak M}}
\nc {\fkN}{{\mathfrak N}}
\nc {\fkO}{{\mathfrak O}}
\nc {\fkP}{{\mathfrak P}}
\nc {\fkQ}{{\mathfrak Q}}
\nc {\fkR}{{\mathfrak R}}
\nc {\fkS}{{\mathfrak S}}
\nc {\fkT}{{\mathfrak T}}
\nc {\fkU}{{\mathfrak U}}
\nc {\fkV}{{\mathfrak V}}
\nc {\fkX}{{\mathfrak X}}
\nc {\fkY}{{\mathfrak Y}}
\nc {\fkW}{{\mathfrak W}}
\nc {\fkZ}{{\mathfrak Z}}

\rnc{\a}{\fka}
\rnc{\b}{\fkb}
\rnc{\c}{\fkc}
\rnc{\d}{\fkd}
\nc{\e}{\fke}
\nc{\f}{\fkf}
\nc{\g}{\fkg}
\nc{\h}{\fkh}
\rnc{\i}{\fki}
\rnc{\j}{\fkj}
\rnc{\k}{\fkk}
\rnc{\l}{\fkl}
\nc{\m}{\fkm}
\nc{\n}{\fkn}
\rnc{\o}{\fko}
\nc{\p}{\fkp}
\nc{\q}{\fkq}
\rnc{\r}{\fkr}
\nc{\s}{\fks}
\rnc{\t}{\fkt}
\rnc{\u}{\fku}
\rnc{\v}{\fkv}
\nc{\x}{\fkx}
\nc{\y}{\fky}
\nc{\w}{\fkw}
\nc{\z}{\fkz}

\nc{\C}{{\mathbb C}}
\nc{\Q}{\mathbb {Q}}
\nc{\Z}{{\mathbb Z}}
\nc{\N}{{\mathbb N}} 

\nc{\Ad}{\operatorname{Ad}}

\nc{\sym}{\mathfrak{S}} 
\nc{\weyl}{\mathfrak{W}}

\nc{\Tor}{\operatorname{Tor}}
\nc{\Hom}{\operatorname{Hom}}
\nc{\End}{\operatorname{End}}
\nc{\Aut}{\operatorname{Aut}}
\nc{\Ext}{\operatorname{Ext}}
\nc{\Coker}{{\operatorname{Coker}}}
\nc{\coker}{{\operatorname{coker}}}
\nc{\Ker}{{\operatorname{Ker}}}
\nc{\id}{{\operatorname{id}}}
\nc{\Id}{\operatorname{Id}}
\nc{\aff}{{\sf aff}}
\nc{\aA}{\wh{A}}

\nc{\rank}{{\operatorname{rank}}}
\nc{\Rep}{{\operatorname{Rep}}}
\nc{\Repfd}{{\operatorname{Rep}_{\sf fd}}}

\nc{\Mod}{{\operatorname{Mod}}}
\nc{\Modfd}{{\operatorname{Mod}_{\sf fd}}}

\nc{\al}{\alpha}
\nc{\ga}{\gamma}
\nc{\del}{\delta}
\nc{\eps}{\epsilon}
\nc{\veps}{\varepsilon}
\nc{\ze}{\zeta}
\nc{\ka}{\kappa}
\nc{\la}{\lambda}
\nc{\si}{\sigma}
\nc{\om}{\omega}

\nc{\Del}{\Delta}

\nc{\ext}{\mathsf{ext}}

\nc{\fksl}{\mathfrak{sl}}
\nc{\fkgl}{\mathfrak{gl}}

\nc{\bsF}{\bbF} 

\nc{\Ons}{\mathbf{O}_q}
\nc{\aOns}{\mathbf{O}_q^{\mathsf{a}}}
\nc{\bOns}{\mathbf{O}_q^{\mathsf{b}}}

\nc{\km}{\mathbf{k}}
\nc{\qkm}{\mathfrak{k}}

\nc{\texp}[1]{\wt{s}_{#1}} 
\nc{\Lus}[1]{T_{#1}} 
\nc{\mLus}[1]{\mathbf{T}_{#1}} 

\rnc{\aff}[1]{\wh{#1}}

\nc{\de}[1]{\epsilon_{#1}} 

\nc{\fIS}{I} 
\nc{\aIS}{\aff{I}} 

\nc{\rt}[1]{\alpha_{#1}} 
\nc{\cort}[1]{h_{#1}}
\nc{\fwt}[1]{\varpi_{#1}}
\nc{\fcwt}[1]{\Lambda_{#1}^\vee}

\nc{\rootsys}{\Delta}

\nc{\hrt}{\vartheta} 

\nc{\cp}[2]{#1(#2)}
\nc{\iip}[2]{\left( #1 , #2\right)} 

\nc{\drv}[1]{\delta_{#1}} 
\nc{\codrv}[1]{d_{#1}} 

\nc{\bsfld}{\mathbb{K}} 

\nc{\central}[1]{c_{#1}}

\nc{\Qlat}{\mathsf{Q}}
\nc{\Qpm}{\Qlat_\pm}
\nc{\Qp}{\Qlat_+}
\nc{\Qm}{\Qlat_-}
\nc{\Qlatv}{\Qlat^\vee}
\nc{\Qvpm}{\Qlat^\vee_{\pm}}
\nc{\Qvp}{\Qlat^\vee_+}
\nc{\Qvm}{\Qlat^\vee_-}
\nc{\aQ}{\aff{\Qlat}}
\nc{\aQv}{\aff{\Qlat}^{\vee}}
\nc{\aQp}{\aff{\Qlat}_+}
\nc{\aQextp}{\aff{\Qlat}_{{\sf ext},+}}
\nc{\aQvp}{\aff{\Qlat}^{\vee}_+}
\nc{\aQvext}{\wt{\Qlat}^{\vee}} 
\nc{\aQvextt}{\aff{\Qlat}^{\vee}_{{\sf ext,\tau}}}
\nc{\aQvextp}{\aff{\Qlat}^{\vee}_{{\sf ext},+}}
\nc{\Plat}{\mathsf{P}}
\nc{\Platv}{\mathsf{P}^\vee}
\nc{\Ppm}{\mathsf{P}_\pm}
\nc{\Pp}{\mathsf{P}_+}
\nc{\Pm}{\mathsf{P}_-}
\nc{\Pvpm}{\mathsf{P}^\vee_\pm}
\nc{\Pvp}{\mathsf{P}^\vee_+}
\nc{\Pvm}{\mathsf{P}^\vee_-}
\nc{\aP}{\wh{\Plat}}
\nc{\aPext}{\wh{\Plat}_{{\sf ext}}}
\nc{\aPpm}{\aP_{\pm}}
\nc{\aPextpm}{\aP_{{\sf ext},\pm}}
\nc{\aPd}{{\aP/\bbZ\drv{}}} 

\nc{\PZ}{{\Plat_{\Z}}}
\nc{\PvZ}{{\Platv_{\Z}}}

\nc{\cl}{\mathsf{cl}}

\nc{\Pcl}{\Plat_{\cl}}
\nc{\Pvcl}{\Plat^{\vee}_{\cl}}

\nc{\Pclz}{\Plat_{\cl,0}}

\nc{\fr}{_\mathsf{fr}}
\nc{\Qfr}{\Qlat\fr}
\nc{\Pfr}{\Plat\fr}

\nc{\Qvfr}{\Qlatv\fr}
\nc{\Pvfr}{\Platv\fr}

\nc{\wgt}[1]{\operatorname{wt}(#1)}

\nc{\gKM}{\g_{\scsop{KM}}}

\nc{\Lg}{L\g} 
\nc{\ag}{\wt{\g}} 
\nc{\agp}{\wh{\g}} 
\nc{\ah}{\wt{\h}} 
\nc{\ahp}{\wh{\h}} 
\nc{\anp}{\wh{\n}^+}
\nc{\anm}{\wh{\n}^-}

\nc{\UqgKM}{U_q(\g_{\scsop{KM}})}

\nc{\Uqg}{U_q(\g)}
\nc{\Uqgp}{U^\prime_q(\g)}

\nc{\Uqb}{U_q{(\b)}}
\nc{\Uqbp}{U_q({\b}^+)}
\nc{\Uqbm}{U_q({\b}^-)}
\nc{\Uqbpm}{U_q({\b}^{\pm})}
\nc{\Uqn}{U_q({\n})}
\nc{\Uqnm}{U_q({\n}^-)}
\nc{\Uqnp}{U_q({\n}^+)}
\nc{\Uqnpm}{U_q({\n}^{\pm})}

\nc{\Uqag}{U_q(\wt{\g})}
\nc{\Uqagp}{U_q(\wh{\g})}

\nc{\Uqan}{U_q(\wh{\n})}
\nc{\Uqanp}{U_q(\wh{\n}^+)}
\nc{\Uqanm}{U_q(\wh{\n}^-)}
\nc{\Uqanpm}{U_q(\wh{\n}^{\pm})}

\nc{\Uqah}{U_q(\wt{\h})}
\nc{\Uqahp}{U_q(\wh{\h})}

\nc{\Uqab}{U_q(\wt{\b})}
\nc{\Uqabp}{U_q(\wt{\b}^+)}
\nc{\Uqabm}{U_q(\wt{\b}^-)}
\nc{\Uqabpm}{U_q(\wt{\b}^{\pm})}
\nc{\Uqabpp}{U_q(\wh{\b}^+)}
\nc{\UqLg}{U_q(\Lg)}

\nc{\CUqag}[1]{\End(\mathsf{f}_{\cO}^{#1})}
\nc{\CUqagint}[1]{\End(\mathsf{f}_{\cO^{\sf int}}^{#1})}

\nc{\CUqLg}[1]{\End(\mathsf{f}_{\cC}^{#1})}

\nc{\QL}[1]{U_q(L{#1})}

\nc{\Kg}[1]{K_{#1}}
\nc{\Kgpm}[1]{K_{#1}^{\pm}}
\nc{\Kgp}[1]{K_{#1}^+}
\nc{\Kgm}[1]{K_{#1}^-}
\nc{\Eg}[1]{E_{#1}}
\nc{\Fg}[1]{F_{#1}}
\nc{\egp}[1]{e^{+}_{#1}}
\nc{\egm}[1]{e^{-}_{#1}}
\nc{\egpm}[1]{e^{\pm}_{#1}}
\nc{\egmp}[1]{e^{\mp}_{#1}}

\nc{\Ce}{\cC}

\nc{\xpm}[1]{x^{\pm}_{#1}}
\nc{\xmp}[1]{x^{\mp}_{#1}}
\nc{\xp}[1]{x^{+}_{#1}}
\nc{\xm}[1]{x^{-}_{#1}}
\nc{\xz}[1]{\xi_{#1}}

\nc{\axzp}[1]{\Psi^{+}_{#1}}
\nc{\axzm}[1]{\Psi^{-}_{#1}}
\nc{\axzpm}[1]{\Psi^{\pm}_{#1}}
\nc{\axze}[1]{\Psi^{\varepsilon}_{#1}}

\nc{\phipm}[1]{\phi^{\pm}_{#1}}
\nc{\phip}[1]{\phi^{+}_{#1}}
\nc{\phim}[1]{\phi^{-}_{#1}}

\nc{\QLp}[1]{U_q(L{#1})^+}
\nc{\QLm}[1]{U_q(L{#1})^-}
\nc{\QLpm}[1]{U_q(L{#1})^\pm}
\nc{\QLz}[1]{U_q^c(L{#1})}

\nc{\chev}{\omega}

\nc{\weightspace}[2]{{{#1}_{#2}}}
\nc{\wsp}[2]{\weightspace{#1}{#2}}
\nc{\wts}[1]{\mathsf{wt}(#1)}
\nc{\hwL}[1]{L(#1)}
\nc{\catO}[3]{\O_{#1}^{#2}{#3}}

\nc{\ev}[1]{\mathsf{ev}_{#1}}
\nc{\qstr}[2]{\Sigma_{#1,#2}}

\nc{\evrep}[2]{V_{#1}(#2)}
\nc{\shrep}[2]{{#1}(#2)}
\nc{\Lshrep}[2]{\Lfml{#1}{#2}}
\nc{\Pshrep}[2]{\Pfml{#1}{#2}}

\nc{\qstrep}[2]{V(\qstr{#1}{#2})}
\nc{\qsrep}[1]{V(#1)}

\nc{\HL}[1]{\mathcal{C}_{#1}}

\nc{\Oinf}{{\cO_\infty}}
\nc{\Oint}{{\cO_\infty^{\sf int}}}

\nc{\ModfdUqLg}{\mathcal{C}} 

\nc{\FF}[2]{\mathsf{f}_{#1}^{#2}}

\nc{\shift}[1]{\Sigma_{#1}}
\nc{\tshift}[1]{\Sigma^{\tau}_{#1}}
\nc{\shiftm}[2]{{#1}_{#2}}

\nc{\shifta}[1]{{\chi}_{#1}}

\nc{\Deltaop}{\Delta^{\sf op}}

\nc{\fml}[2]{{#1}[\negthinspace[#2]\negthinspace]} 
\nc{\Lfml}[2]{{#1}(\negthinspace(#2)\negthinspace)} 
\nc{\Pfml}[2]{{#1}\{#2\}} 

\nc{\qsl}[1]{U_q({\mathfrak{sl}}_{#1})}
\nc{\qasl}[1]{U_q(\wh{\mathfrak{sl}}_{#1})}
\nc{\qlsl}[1]{U_q(L\mathfrak{sl}_{#1})}

\nc{\UqL}[1]{U_q(L{#1})}

\nc{\brS}[1]{S_{#1}} 
\nc{\brSg}[1]{\widetilde{s}_{#1}} 

\nc{\Br}[1]{\mathscr{B}_{#1}} 

\nc{\qWS}[1]{S_{#1}}
\nc{\LT}[1]{T_{#1}}

\nc{\tcorr}[1]{\bm{\xi}_{#1}} 
\nc{\bt}[1]{\mathcal{S}_{#1}} 

\nc{\adt}[1]{\mathcal{T}_{#1}} 
\nc{\adbt}[1]{\mathcal{T}_{#1}} 

\nc{\Rcorr}[1]{\eta_{#1}}

\nc{\intg}{\mathbf{W}^{\mathsf{int}}}
\nc{\Vect}[1]{\operatorname{Vect}_{#1}}

\nc{\corank}{\operatorname{corank}}

\nc{\gsat}[1]{\mathsf{GSat}(#1)} 
\nc{\auxgsat}[2]{\mathsf{Sat}(#1;#2)} 
\nc{\sat}[1]{{\mathbf{S}}} 
\nc{\tsat}{\theta} 
\nc{\zsat}{\zeta} 
\nc{\tsatq}{\tsat_{q}} 
\nc{\zsatq}{\zsat_{q}} 

\nc{\tinv}[1]{\theta_{#1}}

\nc{\oi}{\mathsf{oi}} 

\nc{\Parsetc}{\bm{\Gamma}} 
\nc{\Parsets}{\bm{\Sigma}}

\nc{\Parc}{\bm{\gamma}} 
\nc{\Pars}{\bm{\sigma}}

\nc{\parc}[1]{\bm{\gamma}_{#1}} 
\nc{\pars}[1]{\bm{\sigma}_{#1}}

\nc{\ctheta}{\theta} 

\nc{\qtheta}{\theta_q} 
\nc{\qthetat}{\widetilde{\theta}_q} 

\nc{\Uqk}{U_q(\mathfrak{k})}
\nc{\Uqh}{U_q(\mathfrak{h})}

\nc{\wt}{\widetilde}

\nc{\Ieq}{\bfI_{\sf eq}} 
\nc{\Idiff}{\bfI_{\sf diff}} 
\nc{\Ins}{\bfI_{\sf ns}} 
\nc{\aIeq}{\aIS_{\sf eq}} 
\nc{\aIdiff}{\aIS_{\sf diff}} 
\nc{\aIns}{\aIS_{\sf ns}} 

\nc{\bg}[1]{b_{#1}} 
\nc{\Bg}[1]{B_{#1}} 

\nc{\QK}[1]{\Upsilon_{#1}} 
\nc{\RM}[1]{{R}_{#1}} 
\nc{\QR}[1]{{\Xi}_{#1}} 
\nc{\sRM}[2]{{{R}}_{#1}(#2)} 
\nc{\sRMv}[2]{{\widecheck{R}}_{#1}(#2)} 
\nc{\rRM}[2]{{\mathbf{R}}_{#1}(#2)} 
\nc{\rRMC}[3]{{\mathbf{R}}^{#1}_{#2}(#3)} 
\nc{\rRMv}[2]{\widecheck{\mathbf{R}}_{#1}(#2)} 
\nc{\rRMt}[2]{\widetilde{\mathbf{R}}_{#1}(#2)} 
\nc{\KM}[1]{K_{#1}} 
\nc{\sKM}[2]{K_{#1}(#2)} 
\nc{\rKM}[2]{\mathbf{K}_{#1}(#2)} 
\nc{\trKM}[2]{\wt{\mathbf{K}}_{#1}(#2)} 

\nc{\zeroKM}[2]{K^0_{#1}(#2)} 
\nc{\inftyKM}[2]{K^{\infty}_{#1}(#2)} 

\nc{\tsRM}[2]{{{R}}^{\tau}_{#1}(#2)} 
\nc{\tsRMv}[2]{{\widecheck{R}}^{\tau}_{#1}(#2)} 

\nc{\auxsat}[1]{\bfT} 

\nc{\hext}[2]{{#1}[\negthinspace[#2]\negthinspace]} 

\nc{\binomb}[2]{\genfrac{[}{]}{0pt}{}{#1}{#2}}

\nc{\GQSP}{\cG_{\tsat,\parc{}}}

\nc{\coloneqq}{=}

\nc{\Supp}{\mathsf{Supp}}
\nc{\twistrep}[2]{{#1}^{#2}}  

\nc{\wtbeta}{\wt{\beta}}

\nc{\Gg}{\cG} 
\nc{\gau}{{\bf g}} 

\nc{\nA}{\mathcal{A}} 

\AtBeginDocument{%
   \def\MR#1{}
}

\title[Trigonometric K-matrices]{Trigonometric K-matrices for finite-dimensional representations of quantum affine algebras}

\author[A. Appel]{Andrea Appel} 
\address{Dipartimento di Scienze Matematiche, 
	Fisiche e Informatiche, 
	INdAM GNSAGA and INFN Gruppo Collegato di Parma, 
	Universit\`a di Parma, 
	Parco Area delle Scienze 53/A, 
	43124 Parma (PR), Italy}
\email{\href{mailto:andrea.appel@unipr.it}{andrea.appel@unipr.it}}

\author[B. Vlaar]{Bart Vlaar}
\address{
Beijing Institute of Mathematical Sciences and Applications, No. 544, Hefangkou Village, Huaibei Town, Huairou, Beijing, China, and Max Planck Institute for Mathematics, Vivatsgasse 7, 53111 Bonn, Germany}
\email{\href{b.vlaar@bimsa.cn}{b.vlaar@bimsa.cn}}

\thanks{The first author is partially supported by an INdAM Project 2024 and the PRIN grant 2022HMTBLL}



\begin{document}
	

\begin{abstract}
Let $\g$ be a complex simple finite-dimensional Lie algebra and $\Uqagp$ the corresponding quantum affine algebra. 
We prove that every irreducible finite-dimensional $\Uqagp$-module gives rise to a family of trigonometric solutions of Cherednik's generalized reflection equation. 
These depend upon the choice of a quantum affine symmetric pair $\Uqk\subset\Uqagp$.
Our result relies on the construction of universal K-matrices for arbitrary quantum symmetric pairs we obtained in \cite{AV22} as well as the fact that every irreducible $\Uqagp$-module is generically irreducible under restriction to $\Uqk$. 
In the case of small modules and Kirillov-Reshetikhin modules, we obtain new solutions of the standard and the transposed reflection equations.
\end{abstract}


\maketitle

\setcounter{tocdepth}{1}
\tableofcontents


\section{Introduction}

\subsection{} 
The reflection equation was originally introduced by Cherednik in \cite{Che84} as a consistency condition for factorized scattering on the half-line and was pivotal in the seminal works of Sklyanin \cite{Skl88} and Olshanski \cite{Ols90}.
In the more general form introduced in \cite[Sec.~4]{Che92}, it reads
\begin{equation} \label{eq:CRE-intro}
	\begin{aligned}
		& \rRMC{--}{}{\tfrac{w}{z}}_{21} \cdot (\id_V \ten \rKM{}{w}{}) \cdot \rRMC{-+}{}{zw} \cdot (\rKM{}{z} \ten \id_V) =\\
		& \qquad \qquad = \, (\rKM{}{z} \ten \id_V) \cdot \rRMC{-+}{}{zw}_{21} \cdot (\id_V \ten \rKM{}{w}) \cdot \rRMC{++}{}{\tfrac{w}{z}}\,,
	\end{aligned}
\end{equation}
where $\rKM{}{z}$ is an operator on a finite-dimensional vector space $V$, which depends on a parameter $z\in\bbC^\times$, and $\rRMC{--}{}{z}$, $\rRMC{-+}{}{z}$, and $\rRMC{++}{}{z}$ are operators on $V\ten V$, which satisfy certain {\em mixed} Yang-Baxter equations. 
The operator $\rKM{}{z}$ is often referred to as a K-matrix.

The two most common variants of the reflection equation are recovered as special cases. Given a solution $\rRM{}{z}$ of the spectral Yang-Baxter equation on $V$, the choice $\rRMC{\pm\pm}{}{z}=\rRM{}{z}$ produces the \emph{standard} reflection equation from \cite{Che84, Skl88}.
On the other hand, if $\sf t_1$ and $\sf t_2$ denote the transpositions with respect to the corresponding tensor factors, the choice $\rRMC{--}{}{z}=\rRM{}{z}^{\sf t_1\sf t_2}$, $\rRMC{-+}{}{z}=( \rRM{}{z}^{-1})^{\sf t_1}$, and $\rRMC{++}{}{z}=\rRM{}{z}$ yields the \emph{transposed} reflection equation from \cite{Ols90}.\footnote{
	These two forms of the reflection equation are also known as the \emph{untwisted} and \emph{twisted} reflection equation, 
	respectively. In the context of this paper, such terminology may be misleading, thus we opt for ``standard'' and ``transposed''.
}

\subsection{} 
The reflection equation soon proved ubiquitous in quantum integrability and in representation theory, with further advances in \eg \cite{FM91,KS92,MN92,GZ94,LM94,JKKMW,BPO96,FSHY,IK97,MN98}.
In the new millennium, its study has continued to thrive.
It is essential in the study of coideal subalgebras in Yangians and quantum affine algebras \cite{MR02, DG02, MRS03, DM03, BB10, DLMR03, CGM14, GRW17} and connects to affine Hecke and Temperley-Lieb algebras \cite{Doi05, KM06, IO07, DGN08}.  
Within the context of quantum integrable systems with boundaries, it plays a key role in the diagonalization of Hamiltonians \cite{DGP04, FNR07, Koj11}, in the theory of Baxter Q-operators \cite{YNZ06,YZ06,FS15, BT18, VW20}, and in the theory of quantum KZ equations \cite{ZJ07, SV15, RSV18}.
More recently, it appeared in connection with Schubert calculus \cite{HKZJ20}, three-dimensio\-nal integrability \cite{KP18,KOY19} and gauge theory \cite{BS19,BS20}.

\subsection{} 
In this paper, we solve a central problem which has been open since the introduction of the reflection equation: we provide a systematic, universal approach to the construction of \emph{trigonometric} K-matrices for irreducible finite-dimensional modules over quantum affine algebras. 
Moreover, we obtain a characterization of trigonometric K-matrices in terms of twisted intertwiners, confirming the expectations from \cite{DG02,DM03}.

Our approach requires the extension of the construction of trigonometric R-matrices carried out in \cite{Dri86, FR92, KS95} to the setting of  affine quantum symmetric pair (QSP) subalgebras. 
It relies on our recent work \cite{AV22}, where we prove the existence of {\em universal} K-matrices for quantum Kac-Moody algebras, building on results by Bao and Wang \cite{BW18} and Balagovi\'{c} and Kolb \cite{BK19}.
We now describe our results in more detail. 

\subsection{}
The universal K-matrices constructed in \cite{AV22} for the quantum Kac-Moody algebra $\UqgKM$ depend on the following additional data: a QSP subalgebra, \ie a coideal subalgebra $\Uqk\subset\UqgKM$ quantizing a (pseudo-)fixed-point Lie subalgebra with respect to a (pseudo-)involution of the second kind \cite{Kol14, RV21}, and a \emph{twisting operator} $\psi$, \ie an element of a distinguished class of algebra automorphisms of $\UqgKM$ determined in part by $\Uqk$ (see~Section~\ref{ss:univ-kmx}). 
For any such pair $(\Uqk,\psi)$, we constructed an operator $K$ on category $\cO$ integrable $\UqgKM$-modules, which satisfies a \emph{generalized} reflection equation
\begin{equation}\label{eq:universal-REtwisted-intro}
	(R^{\psi \psi})_{21} \cdot (1 \ten K) \cdot R^\psi \cdot (K \ten 1) = (K \ten 1) \cdot (R^\psi)_{21} \cdot (1 \ten K) \cdot R,
\end{equation}
where $R$ is the universal R-matrix of $\UqgKM$ and we have introduced the shorthand notations
\[
R^\psi = (\psi \ten \id)(R), \qq R^{\psi \psi} = (\psi \ten \psi)(R).
\]
Moreover, $K$ satisfies the QSP intertwining equation $K\cdot b=\psi(b)\cdot K$ for all $b\in\Uqk$.\\

All universal K-matrices corresponding to a given QSP subalgebra are built out of the \emph{quasi}-K-matrix, an operator originally introduced in \cite{BW18} as the solution of a certain QSP intertwining equation for $\fkg = \fksl_N$. 
In \cite{AV22} we proved that the quasi-K-matrix is a solution of a reflection equation of the form \eqref{eq:universal-REtwisted-intro} for a distinguished twisting operator (see~Theorem~\ref{thm:av-k-mx}). 
New solutions are obtained by gauging simultaneously the quasi K-matrix and the twisting operator (see Section~\ref{ss:univ-kmx}). 
In particular, we solve \eqref{eq:universal-REtwisted-intro} for a much wider class of twisting operators than \cite{BK19}.
This feature is essential to the purposes of this paper.

\subsection{}\label{ss:spectral-intro}
Let $\g$ be a complex simple Lie algebra and $\Uqagp$ the corresponding quantum affine algebra. 
Tensor products of finite-dimensional irreducible $\Uqagp$-modules 
are naturally equipped with a trigonometric R-matrix, \ie a rational solution of the spectral Yang-Baxter equation (see Section~\ref{ss:R-matrix}).
By relying on the results from \cite{AV22} for $\Uqagp$, we adapt the strategy underlying the construction of trigonometric R-matrices to the case of universal K-matrices.

Fix a QSP subalgebra $\Uqk\subseteq\Uqagp$ and a twisting operator $\psi$.
The corresponding universal K-matrix $K$ acts on integrable category $\cO$ modules, but not on a finite-dimensional $\Uqagp$-module $V$ in general. 
To remedy this, we consider the grading shift automorphism $\shift{z}$ on $\Uqagp$ and the corresponding shifted module $\Lshrep{V}{z}$. 
We then show that the shifted universal K-matrix $\shift{z}(K)$ gives rise to an operator $\sKM{V}{z}$ on $\Lshrep{V}{z}$ and yields the following result.

\begin{theoremintro}\label{thm:spectral-k-intro}
	Let $V$ be a finite-dimensional $\Uqagp$-module and denote by $V^\psi$ the $\psi$-twisted module $\psi^*(V)$.
	The universal K-matrix $K$ gives rise to a $\Uqk$-intertwiner
	\begin{equation}\label{eq:spectral-intw-intro}
		\sKM{V}{z}: \Lshrep{V}{z}\to\Lshrep{V^\psi}{z^{-1}}\,,
	\end{equation}
	which satisfies the generalized reflection equation \eqref{eq:CRE-intro} with respect to the R-matrices $\sRM{V^{\psi}\, V^\psi}{z}$, $\sRM{\twistrep{V}{\psi}V}{z}$ and $\sRM{VV}{z}$.
\end{theoremintro}

More precisely, Theorem~\ref{thm:spectral-k-intro} does not hold for every possible choice of $\psi$.
In contrast with the case of the R-matrix, we prove that, for the corresponding universal K-matrix to act on $\Lshrep{V}{z}$ and produce the intertwiner \eqref{eq:spectral-intw-intro}, the twisting operator $\psi$ must be carefully chosen. 
In particular, it has to satisfy\footnote{
	Note that this condition is not satisfied by the twisting operators considered in \cite{BK19}.
} $\shift{z}\circ\psi=\psi\circ\shift{z^{-1}}$.
In this case, the universal reflection equation \eqref{eq:universal-REtwisted-intro} reduces to the desired generalized reflection equation, see Theorem~\ref{thm:spectral-k} for the full details.

\subsection{}\label{ss:rational-intro-1}
In the case of irreducible modules, we obtain a trigonometric K-matrix as a suitable normalization of $\sKM{}{z}$.

\begin{theoremintro}\label{thm:rational-k-intro}
	Let $V$ be an irreducible finite-dimensional $\Uqagp$-module. Then, there is a scalar-valued formal Laurent series $g_{V}(z)$ and an operator $\rKM{V}{z}\in\End(V)(z)$ such that
		\begin{align}\label{eq:rational-decomp-k-intro}
			\sKM{V}{z}= g_{V}(z)\cdot\rKM{V}{z}\, 
		\end{align}
	and $\rKM{V}{z}$ is a solution of the generalized reflection equation 
	\eqref{eq:CRE-intro} with respect to the trigonometric R-matrices
	$\rRM{V^{\psi}\, V^\psi}{z}$, $\rRM{\twistrep{V}{\psi}V}{z}$ and $\rRM{VV}{z}$.\footnote{
		The full statement is given in Theorem~\ref{thm:rational-k}.
	}
\end{theoremintro}
The proof of Theorem~\ref{thm:rational-k-intro} crucially relies on the particular form
of the twisting operator and the following result of {\em generic irreducibility}.

\begin{theoremintro} \label{cor:QSP:irreducible-intro}
	Every irreducible finite-dimensional $\Uqagp$-module is generically irreducible 
	as a $\Uqk$-module.
\end{theoremintro}

More generally, we prove that the module $\Lshrep{V}{z}$ remains irreducible under restriction to certain subalgebras in $\Uqagp$, which we refer to as \emph{modified} nilpotent subalgebras (see Theorem~\ref{thm:deformedN:irreducible} and Corollary~\ref{cor:QSP:irreducible}).
This result is of independent interest and generalizes results for the standard Borel and nilpotent subalgebras obtained in \cite{CG05, Bow07, HJ12}.

\subsection{}
Our results drastically simplify the approach to the reflection equation developed in \cite{MN98,DG02,DM03}, 
which is a boundary analogue of the construction of trigonometric R-matrices via linear equations \cite{KR83,Jim86b}. 
First, it assumes the existence of a twisted QSP intertwiner on a given finite-dimensional irreducible $\Uqagp$-module. 
Then, in order to deduce the reflection equation, it requires the tensor product 
to be generically irreducible as a $\Uqk$-module. 

Instead, the existence of a twisted QSP intertwiner is guaranteed by Theorem~\ref{thm:rational-k-intro} on any irreducible module, its essential 
uniqueness follows from Theorem~\ref{cor:QSP:irreducible-intro},
and
the reflection equation follows independently of the irreducibility of the tensor product. 

\subsection{}
Under the additional assumption that $V^{\psi^2}= V$, we further prove that the trigonometric K-matrix can be normalized to be {\em unitary}, \ie satisfies $\rKM{V}{z}^{-1}=\rKM{V^\psi}{z^{-1}}$. 
The condition $V^{\psi^2}= V$ is not restrictive, since the twisting operator $\psi$ can always be chosen to be an involution.
Through unitarity, we establish a direct relation between the poles of $\rKM{V}{z}$ and the irreducibility of $V$ under restriction to $\Uqk$ (see Propositions~\ref{prop:unitarity} and \ref{prop:unitarity:case2}). 
Namely, if $\rKM{V}{z}$ has a pole at $z=\zeta$, then the restriction of $V$ shifted at $\zeta$ is reducible.

In contrast with the case of the trigonometric R-matrix \cite{AK97, Cha02, Kas02}, the poles of the trigonometric K-matrix may fail to detect the reducibility of $V$. 
In Section~\ref{ex:q-Ons-K}, we describe the trigonometric K-matrices arising from the action of the q-Onsager algebra on the vector representation of $U_q(\wh{\mathfrak{sl}}_2)$. 
Interestingly, the number of poles in this case may vary from two to zero, depending on the chosen embedding of the q-Onsager algebra in $U_q(\wh{\mathfrak{sl}}_2)$.
Whenever the trigonometric K-matrix has less than two poles, it fails to detect reducibility.

\subsection{}\label{ss:KR-intro}
Finally, we prove that our approach produces solutions of the standard and transposed reflection equations on certain irreducible $\Uqagp$-modules.
We consider only \emph{$\tau$-restrictable} QSP subalgebras, \ie those whose associated diagrammatic involution $\tau$ fixes the affine node of the Dynkin diagram. 
With the exception of one case in type ${\sf D}^{(1)}_{2n}$, this is equivalent to $\tau$ being either the identity or the affine extension $\eta_0$ of the opposition involution of the underlying subdiagram of finite type (see Section~\ref{ss:tau-res}).

We then consider two special classes of irreducible $\Uqagp$-modules: small modules and Kirillov-Reshetikhin modules. 
We prove that there are canonical twisting operators such that the corresponding trigonometric K-matrices are solutions of either the standard reflection equation if $\tau=\eta_0$ or the transposed reflection equation if $\tau=\id$ (see Theorem~\ref{thm:sol-std-tr-RE} for the case of small modules and Theorem~\ref{thm:sol-KR-RE} for the case of Kirillov-Reshetikhin modules).\\
 
This result produces a large class of new solutions of the standard and the transposed reflection equations. 
Moreover, it shows that many solutions known in the literature arise from the action of universal K-matrices. 
In particular, we recover the solutions constructed in \cite{RV16} for the vector representation in affine types $\sfA$, $\sfB$, $\sfC$ and $\sfD$ for QSP subalgebras with $\tau(0)=0$ (except for affine type $\sfD_{2n}$ with $n$ even and $\tau\neq\id$) and those recently constructed in \cite{KOW24} for Kirillov-Reshetikhin modules in type $\sfA$ for quasi-split QSP subalgebras with fixed parameters (except for type $\sfA.4$).
We expect that the approach used in Section~\ref{s:recover-RE} naturally extends beyond the case $\tau\in\{\id,\eta_0\}$.

\subsection{Outline}
In Section~\ref{s:affine}, we review the construction of trigonometric R-matrices on irreducible finite-dimensional representations of quantum affine algebras.
In Section~\ref{s:affine-qsp}, 
we outline the construction of universal K-matrices for quantum Kac-Moody algebras arising from QSP subalgebras obtained in \cite{AV22}.
In Section~\ref{s:spectral}, we prove that universal K-matrices for quantum affine algebras give rise to formal spectral operators on finite-dimensional representations. 
In Section~\ref{s:rational}, we prove the existence of trigonometric K-matrices on irreducible representations. 
The latter result relies on generic irreducibility under restriction to a QSP subalgebra, which we prove in Section~\ref{s:irreducibility}. 
Finally, we apply our construction to the case of small modules and Kirillov-Reshetikhin modules in Section~\ref{s:recover-RE}.

\subsection{Acknowledgments}
The authors would like to thank Martina Balagovi\'{c}, Vyjayanthi Chari, Gustav Delius, Pavel Etingof, Sachin Gautam, Stefan Kolb, Tomasz Prze\'{z}dziecki, Nicolai Reshetikhin, Jasper Stokman, Catharina Stroppel, Valerio Toledano Laredo, Curtis Wendlandt, Robert Weston and Huafeng Zhang for their interest in this work and for useful discussions. 
This work started during the hybrid Mini-Workshop \emph{Three Facets of R-Matrices} at the Mathematisches Forschungsinstitut Oberwolfach in October 2021. 
The authors are grateful to the organizers for their invitation and to the entire MFO team for the wonderful working conditions.
BV would also like to thank the Max Planck Institute for Mathematics in Bonn for the great research environment during his visit.


\section{Quantum affine algebras and their R-matrices}\label{s:affine}

In this section, we recall the definition of (untwisted) quantum affine algebras and basic results on their irreducible finite-dimensional modules. Given a free abelian group $\mathsf{\Lambda}=\bigoplus_{\alpha}\bbZ\mathbf{e}_{\alpha}$, we shall denote its non-negative cone by $\mathsf{\Lambda}_+$, \ie $\mathsf{\Lambda}_+=\bigoplus_{\alpha}\bbZ_{\geqslant 0}\mathbf{e}_{\alpha}$.


\subsection{Affine Lie algebras \cite{Kac90}}\label{ss:affine}

Let $\g$ be a finite-dimensional simple Lie algebra defined over
$\bbC$ with Cartan subalgebra $\h\subset\g$. 
Let $\fIS\coloneqq\{1,2,\dots, \rank(\g)\}$ be the set of vertices of the corresponding Dynkin diagram, $A=(a_{ij})_{i,j\in\fIS}$ the Cartan matrix and $\iip{\cdot}{\cdot}$ the normalized invariant bilinear form on $\g$.
Denote the simple roots by $\al_i$ and the simple coroots by $h_i$ ($i \in \fIS$) so that $\cp{\rt{j}}{\cort{i}}=a_{ij}$ for all $i,j \in \fIS$. 
We consider the root and coroot lattices:
\begin{equation*}
\Qlat = {\sf Sp}_\Z \{ \al_i \, | \, i \in \fIS \} \subset\h^*\qq\mbox{and}\qq \Qlat^{\vee} = {\sf Sp}_\Z \{ h_i \, | \, i \in \fIS \} \subset\h.
\end{equation*}
Let $\rootsys_+\subset\Qp$ be the set of positive roots, $\hrt=\sum_{i\in\fIS}a_i\rt{i}$ the highest root, where $a_i$ are the labels of the extended Dynkin diagrams from \cite[Table Aff 1.]{Kac90}, and $\Plat\coloneqq\{\lambda\in\h^*\;|\;\lambda(\Qlatv)\subset\bbZ\}$ the weight lattice.\\

Let ${\ag}$ be the Kac-Moody algebra of (untwisted) affine type 
associated to $\g$ with affine Cartan subalgebra ${\ah}\subset{\ag}$.
Set $\agp\coloneqq [\ag,\ag]$ and $\ahp \coloneqq \agp \cap \ah$.
Let $\aIS\coloneqq\{0\}\cup\fIS$ be the set of vertices of the affine Dynkin diagram 
and $\aff{A}=(a_{ij})_{i,j\in\aIS}$ the extended Cartan matrix, see \cite[Table Aff. 1]{Kac90}. 
We denote by $\aQv \subset {\ah}$ and $\aQ \subset {\ah}^*$ the affine coroot and root lattices, respectively.
Let $\drv{}\in\aQp$ and $c\in\aQvp$ be the unique elements such that
\begin{align*}
	\{\lambda\in\aQ\;|\; \forall i\in\aIS,\, \cp{\lambda}{\cort{i}}=0 \}=\Z\drv{}\aand
	\{h\in\aQv\;|\; \forall i\in\aIS,\, \cp{\rt{i}}{h}=0 \}=\Z \central{}
\end{align*}
In particular, $\drv{}=\rt{0}+\hrt$, $c$ is central in ${\ag}$, and, under the identification $\nu:{\ah}\to{\ah}^*$ induced by the bilinear form, one has $\nu(c)=\delta$.
The sets of real and imaginary affine positive roots in $\aQp$ are described by 
\begin{equation*}
	\aff{\Delta}_+^{\sf re}=\Delta_++\bbZ_{\geqslant0}\drv{}
	\aand
	\aff{\Delta}_+^{\sf im}=\bbZ_{>0}\drv{}\, .
\end{equation*}
We fix $\codrv{}\in{\ah}$ such that $\cp{\rt{i}}{\codrv{}}=\drc{i0}$ for any $i\in\aIS$. 
Note that $\codrv{}$ is defined up to a summand proportional to $c$ and we obtain a natural identification ${\ah}=\h\oplus\bbC c\oplus\bbC\codrv{}$. 
In terms of the \emph{extended} coroot lattice $\aQvext\coloneqq\aQv\oplus\bbZ d\subset{\ah}$ we set 
\[
\aP\coloneqq\{\lambda\in{\ah}^*\;|\;\lambda(\aQvext)\subset\bbZ\}\,.
\]
Then, the quotient lattice
$\aP/\bbZ\drv{}\simeq\hom_{\bbZ}(\aQv,\bbZ)$
has a basis given by the images of the fundamental weights in $\aP$.

\subsection{Drinfeld-Jimbo presentation of the quantum affine algebra} \label{ss:quantumaffine}

Let $q$ be an indeterminate\footnote{
	The results of this paper are also valid in the case $\bsF=\bbC$ and $q\in\bbC^\times$ not a root of unity.
} and set $\bsF\coloneqq\ol{\bbC(q)}$.
Fix non-negative integers $\{\de{i}\;|\; i\in\aIS\}$ such that the matrix $(\de{i}a_{ij})_{i,j\in\aIS}$ is symmetric and set $q_i\coloneqq q^{\de{i}}$.

The quantum Kac-Moody algebra associated to ${\ag}$, see \cite{Dri85,Dri86,Jim86a}, is the unital associative algebra $\Uqag$ defined over $\bsF$ with generators $\Eg{i}$ and $\Fg{i}$ ($i\in\aIS$), and $\Kg{h}$ ($h\in\aQvext$) subject to the following defining relations:
\begin{gather*}
	\Kg{h}\Kg{h'}=\Kg{h+h'}, \qq \Kg{0}=1, \\	
	\Kg{h}\Eg{i}=q^{\cp{\rt{i}}{h}}\Eg{i}\Kg{h}, \qq \Kg{h}\Fg{i}=q^{-\cp{\rt{i}}{h}}\Fg{i}\Kg{h},\\
	[\Eg{i},\Fg{j}]=\drc{ij}\frac{\Kg{i}-\Kg{i}^{-1}}{q_i-q_i^{-1}}, \label{Uqag:relns3} 
\end{gather*}
for any $i,j\in\aIS$ and $h,h'\in\aQvext$, where $\Kg{i}^{\pm1}=\Kg{\pm \de{i}\cort{i}}$,
together with the quantum Serre relations, see, \eg \cite[Cor.~33.1.5]{Lus94},
\begin{equation*} 
	\mathsf{Serre}_{ij}(E_i,E_j) =0= \mathsf{Serre}_{ij}(F_i,F_j)
\end{equation*}
for $i\neq j$, where $\mathsf{Serre}_{ij}$ denotes the following polynomial in two noncommuting variables:
\begin{equation} \label{Serredef}
	\mathsf{Serre}_{ij}(x,y) = \sum_{r=0}^{1-a_{ij}} (-1)^n \binom{1-a_{ij}}{r}_{q_i} x^{1-a_{ij}-r} y x^r
\end{equation}
with the q-deformed binomial coefficient, see, \eg \cite[1.3.3]{Lus94}.
We denote by $\Uqg$ the subalgebra generated by $\Eg{i}, \Fg{i}$ ($i\in\fIS$), and $\Kg{h}$ ($h\in\Qlatv$), which is a quantum Kac-Moody algebra of finite type.\\

We consider the Hopf algebra structure on $\Uqag$ determined by the coproduct formulae
\begin{align}
\Delta(\Eg{i}) &= \Eg{i}\ten1+\Kg{i}\ten\Eg{i},	& \Delta(\Fg{i}) &= \Fg{i}\ten\Kg{i}^{-1}+1\ten\Fg{i},	& \Delta(\Kg{h}) &= \Kg{h}\ten\Kg{h}, 
\end{align}
for any $i\in\aIS$ and $h\in\aQvext$.
The Chevalley involution $\chev:\Uqag\to\Uqag$ is the morphism of algebras defined by
\begin{equation}\label{eq:chevalley}
	\chev(\Eg{i})=-\Fg{i},\qq \chev(\Fg{i})=-\Eg{i}\, \qq \chev(\Kg{h})=\Kg{-h},
\end{equation}
for any $i\in\aIS$ and $h\in\aQvext$. Note that $\omega$ is an isomorphism of Hopf algebras
$\Uqag\to\Uqag^{{\sf cop}}$, where $\Uqag^{{\sf cop}}$ denotes the Hopf algebra structure given by the opposite coproduct $\Delta^{\sf op}$.\\

We denote by $\Uqahp$, $\Uqanp$, and $\Uqanm$ the subalgebras generated, respectively, by $\Kg{\pm h_i}$, $\Eg{i}$, and $\Fg{i}$ ($i\in\aIS$).
Their {finite-type} analogues $\Uqh, \Uqnpm$ are analogously defined, by replacing $\aIS$ with $\fIS$.

We set $\Uqabpm\coloneqq\Uqanpm\Uqah$, where $\Uqah$ is the commutative subalgebra generated by $\Kg{h}$, $h\in\aQvext$. \\

The \emph{quantum affine algebra} is the subalgebra $\Uqagp\coloneqq\Uqanp\Uqahp\Uqanm$, while
the \emph{quantum loop algebra} $\UqLg$ is the quotient of $\Uqagp$ by the ideal generated by $\Kg{c}-1$.
Note that the Hopf algebra structure and the Chevalley involution restrict to $\Uqagp$ and descend to $\UqLg$.

\subsection{Category $\Oinf$ and finite-dimensional modules}\label{ss:QL-representations}
Drinfeld-Jimbo quantum groups have two important categories of modules: finite-dimensional modules and integrable modules in category $\mathcal{O}$. For quantum groups of finite type, the two categories coincide, while they are strikingly different in affine type.

Recall that $V\in\Mod(\Uqah)$ is a (type {\bf 1}) weight module if $V=\bigoplus_{\mu\in\aP}\wsp{V}{\mu}$, where
	\begin{equation}
		\wsp{V}{\mu}\coloneqq\{v\in V\;|\; \forall h\in\aQvext,\, \Kg{h}\cdot v=q^{\cp{\mu}{h}}v\}\, .
	\end{equation}
In particular, $\Uqag$, $\Uqagp$, $\Uqabpm$ and $\Uqanpm$ are weight modules via the action $\Kg{h} \cdot x = \Kg{h} x \Kg{h}^{-1}$.\\

We denote by $\Oinf\subset\Mod(\Uqag)$ the full subcategory of weight modules with a locally finite $\Uqanp$-action.\footnote{
	This category is denoted by $\cC^{\scsop{hi}}$ in \cite[Sec.~3.4.7]{Lus94}. See also its generalization introduced in \cite[Sec.~15.1]{ATL24a}.
}
This is a braided monoidal category (see Section~\ref{ss:R-matrix}). Note that modules in $\Oinf$ are not necessarily finitely generated.

\begin{remark}
	Recall that, in analogy with \cite[Ch.~9]{Kac90}, a $\Uqag$-module is in category $\cO$ if it is 
	a weight module with finite-dimensional weight spaces, whose weights are contained in the union of finitely many cones. 
	Category $\cO$ modules are contained in $\Oinf$. 
	\hfill\rmkend
\end{remark}

Let $\Oint\subset\Oinf$ be the full subcategory of {\em integrable} modules, \ie those equipped with a locally nilpotent action
of the elements $\Eg{i}, \Fg{i}$ ($i\in\aIS$). This is a semisimple braided monoidal category, whose
nontrivial irreducible modules are infinite-dimensional and classified by nonzero dominant weights, see \cite[Thm.~6.2.2, Cor.~6.2.3]{Lus94}.\\

Let $\ModfdUqLg$ denote the category of finite-dimensional (type {\bf 1}) weight modules over $\Uqagp$. Any module in $\ModfdUqLg$ admits a weight decomposition over the quotient lattice $\aPd$, since
the central element $\Kg{c}$ forcibly acts as $1$. In particular,
$\ModfdUqLg$ is the category of finite-dimensional (type {\bf 1}) weight modules over $\UqLg$. 
This category is not semisimple and its irreducible representations are classified by $\rank(\g)$-tuples of monic polynomials, see, \eg \cite[Thm.~12.2.6]{CP95}. 
Moreover, $\ModfdUqLg$ is monoidal, but it is not braided.


\subsection{Completions}\label{ss:completions}
In the following, we consider suitable completions of the algebras $\Uqag$ and $\UqLg$, which contain certain distinguished operators acting on 
$\Oint$ and on $\ModfdUqLg$, respectively.\\

Let $A$ be a bialgebra and $\cA\subseteq\Mod(A)$ a monoidal subcategory. For any $n>0$, we denote by $\End(\FF{\cA}{})$ the algebra of endomorphisms of the $n$-fold forgetful functor $\FF{\cA}{n}\colon\cA^{n}\to\Vect{}$, given by $\FF{\cA}{n}(V_1,\dots, V_n)= V_1\ten\cdots\ten V_n$. Note that there is a canonical morphism of algebras $A^{\ten n}\to \End(\FF{\cA}{n})$.
The monoidal structure on $\cC$ induces on the tower of algebras $\End(\FF{\cA}{n})$ ($n\geqslant1$) the structure of a \emph{cosimplicial algebra}. Roughly, this means that $\End(\FF{\cA}{})$ can be thought of as a topological bialgebra, whose structure extends that of $A$, see \eg~ \cite[Sec.~8.8]{ATL19}.
Every automorphism $\phi:A\to A$, whose pullback preserves $\cA$, naturally extends to $\End(\FF{\cA}{})$.\\

We denote by $\CUqag{n}$ and $\CUqagint{n}$ the completion of $\Uqag^{\ten n}$ with respect to $\Oinf$ and $\Oint$, respectively.
There is a canonical morphism $\CUqag{n}\to\CUqagint{n}$ given by restriction from $\Oinf$ to $\Oint$.
Moreover, by \cite[Sec.~3]{ATL24b}, $\Uqag^{\ten n}$ embeds into both $\CUqag{n}$ and $\CUqagint{n}$.

Similarly, we denote by $\CUqLg{n}$ the completion of $\UqLg^{\ten n}$ with respect to $\ModfdUqLg$. 


\subsection{The universal R-matrix}\label{ss:R-matrix}

The Hopf algebra $\Uqag$ is {quasitriangular}, \ie it admits a universal R-matrix $\RM{}\in\End(\FF{\cO}{2})$, satisfying the intertwining equation $\RM{}\Delta(x)=\Deltaop(x) R$ for any $x\in\Uqag$ and the coproduct identities
\begin{equation}
	(\Delta \ten \id)(\RM{}) = \RM{13} \RM{23},
	\qq \qq
	(\id \ten \Delta)(\RM{}) = \RM{13} \RM{12}\, .
\end{equation}
In particular, it follows that $\RM{}$ satisfies the Yang-Baxter equation 
\begin{equation*}
	\RM{12} \RM{13} \RM{23} = \RM{23} \RM{13} \RM{12}.
\end{equation*}
As a consequence, the tensor category $\Oinf$ is braided. 

The element $\RM{}$ arises from the Drinfeld double construction of $\Uqag$ as the canonical tensor of a Hopf pairing between $\Uqabm$ and $\Uqabp$
\cite{Dri86, Lus94}.
More precisely, let $\{u_i\}, \{u^i\}\subset\h$ be dual bases and set 
\begin{equation*} 
	\Omega_0\coloneqq\sum_iu_i\ten u^i, \qq \qq \wh{\Omega}_0\coloneqq m(c\ten d+d\ten c)+\Omega_0,
\end{equation*}
where $m=1,2,3$ if $\g$ is of type $\mathsf{ADE}$, $\mathsf{BCF}$, or $\mathsf{G}$, respectively. 
Then the R-matrix of $\Uqag$ lies in the completion of $\Uqabm \ten \Uqabp$ with respect to the $\Qlat^+$-grading and has the form
\begin{align}\label{eq:R-mx}
	\RM{} = q^{\wh{\Omega}_0}\cdot\sum_{\mu>0}\Xi_{\mu},
\end{align}
where $\Xi_{\mu}\in\Uqanm_{-\mu}\ten\Uqanp_{\mu}$ 
and $q^{\wh{\Omega}_0}$ 
acts on tensor products of weight vectors as 
\begin{equation*}
q^{\wh{\Omega}_0}\cdot (v \ten w) = q^{\iip{\wgt{v}}{\wgt{w}}} v \ten w.
\end{equation*}
Since $(\omega\ten\omega)(\RM{})=\RM{21}$, the Chevalley involution is an isomorphism of the quasitriangular Hopf algebras $\Uqag$ and $\Uqag^{\sf cop}$.


\subsection{The spectral R-matrix}\label{ss:spectral-R}
The universal R-matrix of $\Uqag$ does not immediately act on finite-dimensional $\UqLg$-modules. 
The first obstacle is given by the operator $q^{m(c\ten d+d\ten c)}$. 
However, this is easily solved by observing that, since the central element $\Kg{c}$ acts by $1$ if and only if $c$ acts by zero, that factor can be ignored.
The second obstacle is given by the fact that the operator $\Xi\coloneqq\sum_{\mu>0}\Xi_{\mu}$ is not necessarily defined on finite-dimensional representations over $\UqLg$. 
To this end, set
\begin{equation*}
	\Uqag[z,z^{-1}]\coloneqq\Uqag\ten\bbF[z,z^{-1}]
\end{equation*}
and consider the \emph{homogeneous grading shift} automorphism
\begin{equation}\label{eq:grading-shift}
	\shift{z}: \Uqag\to\Uqag[z,z^{-1}]
\end{equation}
given by $\shift{z}(\Eg{i})\coloneqq z^{\drc{i0}}\Eg{i}$, $\shift{z}(\Fg{i})\coloneqq z^{-\drc{i0}}\Fg{i}$, and $\shift{z}(\Kg{h})\coloneqq\Kg{h}$. 
Note that, by specializing $z$ to $\bsF^\times$, we obtain a one-parameter family of automorphisms of $\Uqag$. 
Then, let
\begin{equation}\label{eq:shift-coproduct}
	\Delta_{z}, \Deltaop_{z}:\Uqag\to(\Uqag\ten\Uqag)[z,z^{-1}]
\end{equation}
be the \emph{shifted coproducts} defined by 
\begin{equation*}
	\Delta_z(x) \coloneqq (\id\ten\shift{z})(\Delta(x)), \qq \Deltaop_z(x) \coloneqq (\id\ten\shift{z})(\Deltaop(x)).
\end{equation*}
The grading shift is well-defined on $\UqLg$. 
For any $V\in\ModfdUqLg$ with action $\pi_V:\UqLg\to\End(V)$, we consider the infinite-dimensional (over $\bsF$) modules 
\begin{equation*}
	\shrep{V}{z}\coloneqq V\ten\bsF(z), \qq \Lshrep{V}{z}\coloneqq V\ten\Lfml{\bsF}{z},
\end{equation*}
with action given by $\pi_V(\shift{z}(x))$. 
By considering the projection of the formal series 
\begin{equation*}
(\id\ten\shift{z})(\RM{})\in \fml{ (\Uqag^{\ten 2})^{\Oinf}}{z}
\end{equation*}
on the quantum loop algebra, one obtains the following theorem, see~\cite{Dri86, FR92, Her19}.

\begin{theorem}\label{thm:spectral-R}\hfill
	\begin{enumerate}\itemsep0.25cm
		\item The quantum loop algebra $\UqLg$ has a universal \emph{spectral} R-matrix, \ie a distinguished element 
		$\sRM{}{z}\in\fml{\CUqLg{\ten 2}}{z}$ such that $(\shift{a}\ten\shift{b})(\sRM{}{z})=\sRM{}{\frac{b}{a}z}$ ($a,b\in\bsF^{\times}$), 
		\begin{gather*}
			\sRM{}{z}\Delta_{z}(x) = \Deltaop_{z}(x)\sRM{}{z}
		\end{gather*}
		for $x \in \UqLg$, and  following coproduct identities are satisfied:
		\begin{gather*}
			(\Delta_z\ten\id)(\sRM{}{zw}) = \sRM{}{zw}_{13} \sRM{}{w}_{23},\\ (\id\ten\Delta_w)(\sRM{}{z}) = \sRM{}{z}_{13} \sRM{}{zw}_{12} \, .
		\end{gather*}
		In particular, the spectral Yang-Baxter equation holds:
		\begin{equation*}\label{eq:z-YBE}
			\sRM{}{z}_{12} \sRM{}{zw}_{13} \sRM{}{w}_{23} = \sRM{}{w}_{23} \sRM{}{zw}_{13} \sRM{}{z}_{12}\, .
		\end{equation*}
		\item For any $V,W\in\ModfdUqLg$, the operator
		\begin{equation}\label{eq:sRM-rep}
			\sRM{VW}{z} \coloneqq (\pi_{V}\ten\pi_W)(\sRM{}{z})
			\in\fml{\End(V\ten W)}{z}
		\end{equation}
		is well-defined and yields an intertwiner
		\begin{equation}\label{eq:sRMv-rep}
			\sRMv{VW}{z}\coloneqq(1\,2)\circ\sRM{VW}{z}: V \ten \Lshrep{W}{z} \to \Lshrep{W}{z} \ten V \, .
		\end{equation}
	\end{enumerate}
\end{theorem}

\begin{remark}
	The specialization of $z$ in $\bsF^\times$ yields the notion of a shifted
	finite-dimensional $\UqLg$-module. Namely, for any $a\in\bsF^\times$
	and $V\in\ModfdUqLg$ with action map $\pi_V:\UqLg\to\End(V)$, 
	the shifted $\UqLg$-module $\shrep{V}{a}$ is the vector space $V$ 
	equipped with the action $\pi_V\circ\shift{a}$.\hfill\rmkend
\end{remark}


\subsection{Trigonometric R-matrices}\label{ss:rational-R}
In the case of irreducible modules, the operator $\sRM{VW}{z}$ has the following rationality property, see, \eg~\cite{Dri86, FR92, KS95} and \cf~\cite{Jim86b}. 

\begin{theorem}\label{thm:rational-R}
	Let $V,W\in\ModfdUqLg$ be two irreducible representations.
	There exists a canonical 
	scalar-valued formal Laurent series $f_{VW}(z)\in\Lfml{\bsF}{z}$ such that
	\begin{equation*}
		\rRM{VW}{z}\coloneqq f_{VW}(z)^{-1}\sRM{}{z}\in\Lfml{\End(V\ten W)}{z}
	\end{equation*} 
	is rational, satisfies the spectral Yang-Baxter equation \eqref{eq:z-YBE} and satisfies the unitarity relation 
	\begin{equation} \label{eq:R:unitarity}
		\rRM{VW}{z}^{-1}=(1\, 2)\circ\rRM{WV}{z^{-1}}\circ(1\, 2).
	\end{equation}
	In particular, $\rRMv{VW}{z}\coloneqq(1\,2)\circ\rRM{VW}{z}$ 
	is an intertwiner $V\ten\shrep{W}{z}\to \shrep{W}{z}\ten V$.
\end{theorem}

The operator $\rRM{VW}{z}$ is called the \emph{normalized} R-matrix, which is an example of a \emph{trigonometric} R-matrix. 
The proof of the theorem relies on the \emph{generic irreducibility} of the tensor product $V\ten W$, \ie on the irreducibility of $V\ten\Lfml{W}{z}$ over $\Lfml{\UqLg}{z}$ (see~ \cite[Sec.~4.2]{KS95} or \cite[Thm.~3]{Cha02}). 
Note that the function $f_{VW}(z)$ is uniquely determined by the condition $\rRM{}{z}(v_0\ten w_0)=v_0\ten w_0$, where $v_0\in V$ and $w_0\in W$ are $\ell$-{\em highest weight vectors}.\footnote{
	By \cite[Cor.~12.2.5]{CP95}, every finite-dimensional irreducible module $V$ is generated by an $\ell$-{\em highest weight vector}, \ie a weight 			vector $v_0$, which is annihilated by $\Fg{0}$ and $\Eg{i}$ ($i\in\fIS$) and is a simultaneous eigenvector for all imaginary root vectors.
}

\begin{remarks}\label{rmk:spectral-R-matrix}
	\hfill
	\begin{enumerate}[leftmargin=2em]\itemsep0.25cm
		\item  \label{rmk:Rcrossingsymmetry}
		If $q \in \C^\times$ is not a root of unity, $\sRM{VW}{z}$ is a \emph{meromorphic} operator for 
		{\em any} finite-dimensional modules $V$ and $W$ \cite{FR92,KS95, EM02}.  
		Relying on the functional relation between $\sRM{VW}{z}$ and $\sRM{V^{**} \, W}{z}$ (known as \emph{crossing symmetry}), 
		one proves that the operator $\sRM{}{z}$ is analytic near zero and therefore meromorphic on $\bbC$.
		\item  \label{rmk:Rregularity}
		A finite-dimensional irreducible $\UqLg$-module $V$ is called	{\em real} if $V\ten V$ is still irreducible. 
		Thus, in this case, one has $\rRMv{VV}{1}=\id_{V\ten V}$ (see also \cite[Lemma 10.2 and the following discussion]{FHR21}). 
		\hfill \rmkend
	\end{enumerate}
\end{remarks}

\begin{example}
	Let $V_1=\bsF^2$ be the fundamental representation of $U_q(\mathfrak{sl}_2)$. For any $a\in\bsF^\times$, we 
	consider the evaluation representation $\evrep{1}{a}=\bsF^2$
	with action of $U_q(L\mathfrak{sl}_2)$ given by 
	\begin{gather*}
		\pi(E_0) = \begin{pmatrix} 0 & 0 \\ q^{-1}a & 0 \end{pmatrix} = q^{-1}a\pi(F_1)\, , 
		\qq \qq
		\pi(F_0) = \begin{pmatrix} 0 & qa^{-1} \\ 0 & 0 \end{pmatrix}= qa^{-1}\pi(E_1)\,, \\
		\pi(K_0) = \begin{pmatrix} q^{-1} & 0 \\ 0 & q \end{pmatrix} =\pi(K_1)^{-1}\,.
	\end{gather*}
	In the case of $\evrep{1}{a}\ten\evrep{1}{bz}$, the rational function $\rRM{ab}{z}\coloneqq\rRM{\evrep{1}{a}\,\evrep{1}{b}}{z}$ is easily computed, see, \eg \cite[12.5.7]{CP95}, \cite{Jim86b}.
	Set $\lambda\coloneqq b/a$. 
	Then
	\begin{equation*}
		\rRM{ab}{z}\coloneqq
		\left(
		\begin{array}{cccc}
			1 & 0 & 0 & 0\\
			0 & \frac{q(1-\lambda z)}{q^2-\lambda z} & \frac{\lambda z(q^2-1)}{q^2-\lambda z} & 0\\
			0 & \frac{q^2-1}{q^2-\lambda z} & \frac{q(1-\lambda z)}{q^2-\lambda z} & 0\\
			0 & 0 & 0 & 1
		\end{array}
		\right).
	\end{equation*}
	Note that, if $\lambda=q^{2}$, $\rRM{ab}{z}$ has a pole at $z=1$, while, if $\lambda=q^{-2}$, $\rRM{ab}{z}$ is not invertible at $z=1$. 
	It is well-known that $\evrep{1}{a}\ten\evrep{1}{b}$ fails to be irreducible precisely when $\la=q^{\pm 2}$. 
	Finally, note that $\rRMv{aa}{1}$ is the identity.
	\rmkend
\end{example}


\section{Quantum affine symmetric pairs}\label{s:affine-qsp}

\subsection{Generalized Satake diagrams}\label{ss:gsat}

Classical and quantum Kac-Moody algebras are defined in terms of combinatorial datum encoded by the Dynkin diagram and the Cartan matrix. 
Similarly, classical and quantum symmetric pairs \cite{Let02,Kol14} 
arise from a refinement of such datum.\\

Let $\Aut(\aff{A})$ be the group of \emph{diagram automorphisms} of the affine Cartan matrix, \ie the group of permutations $\tau$ of $\aIS$ such that $a_{ij}=a_{\tau(i)\tau(j)}$.
Let $X \subset\aIS$ be a proper subset of indices. 
Thus, the corresponding Cartan matrix $\aA_X$ is necessarily of finite type.
We denote by $\ag_X\subset\ag$ the corresponding Lie subalgebra and by $\oi_X\in\Aut(\wh{A}_X)$ the {\em opposition involution of $X$}, \ie the involutive diagram automorphism of $X$ induced by the action of the longest element $w_X$ of the Weyl group $W(\ag_X)$ on the root lattice of $\ag_X$.

\begin{definition}[\cite{RV20,RV21}] \label{def:GSat}
	A \emph{generalized (affine) Satake diagram} is a pair $(X,\tau)$ where $X \subset\aIS$ and $\tau$ is an involutive diagram automorphism stabilizing $X$ such that
	\begin{enumerate}\itemsep0.25cm
		\item $\tau|_X = \oi_X$;
		\item for any $i \in \aIS \setminus X$ such that $\tau(i)=i$, the connected component of $X \cup \{ i \}$ containing $i$ is not of type ${\sf A}_2$.
	\end{enumerate}
	The set of all such diagrams is denoted by $\gsat{\aff{A}}$. \hfill \rmkend
\end{definition}

A classification of generalized Satake diagrams of affine type is provided in \cite[App.~A, Tables 5, 6 and 7]{RV21}.
Henceforth, we fix $(X,\tau)\in \gsat{\aff{A}}$.

\begin{example} \label{exam:sl2GSat}
	Consider the affine Lie algebra $\wh{\mathfrak{sl}}_2$ and set $\wh I = \{ 0,1 \}$.
	There are four generalized Satake diagrams given by
\[	
(\emptyset,\id)\,, \qq (\emptyset,(01))\,, \qq (\{0\},\id)\,, \qq  (\{1\}, \id) 
\]
where $(01)$ denotes the nontrivial diagram automorphism. \hfill \rmkend
\end{example}

\subsection{Pseudo-involutions}\label{ss:pseudo-involutions}
The diagram automorphism $\tau\in\Aut(\aA)$ extends canonically to an automorphism of ${\agp}$, given on the generators by $\tau(e_i)=e_{\tau(i)}$, $\tau(f_i)=f_{\tau(i)}$, and $\tau(\cort{i}) = \cort{\tau(i)}$. 
The pair $(X,\tau)$ is then associated to the Lie algebra automorphism $\theta:\agp\to\agp$, given by
\begin{align}\label{def:theta}
	\theta \coloneqq \Ad(\wt{w}_X) \circ \omega \circ \tau
\end{align}
where $\omega$ denotes the Chevalley involution on $\aff\g$ 
and $\wt{w}_X$ is defined in terms of Tits' triple exponential operators, see, \eg \cite[Sec.~3.11]{AV22}.
Note that $c$ is fixed by $\tau$ and $w_X$.
Hence, $\tsat(c)=-c$ and $\tsat$ descends to an automorphism of $L\fkg$.

\begin{remarks} \label{rmk:KM:pseudoinvolution} \hfill
	\begin{enumerate}[leftmargin=2em]\itemsep0.25cm
		\item \label{rmk:KM:pseudoinvolution:1}
		In \cite[Sec.~4.9]{KW92}, Kac and Wang defined a canonical procedure (for arbitrary generalized Cartan matrices) to extend a 
		diagram automorphism from ${\ahp}$ to ${\ah}$. 
		Thus, $\theta$ extends to an automorphism of ${\ag}$, which is \emph{of the second kind} (see \cite[4.6]{KW92}) and is an 
		involution on $\ah$.
		Following \cite{RV21}, we shall refer to $\tsat$ as a \emph{pseudo-involution of ${\ag}$ of the second kind}. 
		Note that ${\ah}^{\theta} \subseteq{\ahp}$, see \cite[Sec.~6.2]{AV22}.
		\item \label{rmk:KM:pseudoinvolution:2}
		Since $(X,\tau)$ and therefore $\theta$ are fixed, in the following we 
		will use the subscript $\tsat$ even in the case of objects explicitly defined 
		in terms of $(X,\tau)$. 
		Note that the datum $(X,\tau)$ can be recovered from $\theta$ 
		since $X=\{ i \in \aIS \, | \, \theta(h_i)=h_i \}$ and $\tau=\omega \circ \Ad(\wt{w}_X)^{-1} \circ \theta$. 
		\ourcomment{No bijection between $\gsat{\aff{A}}$ and pseudo-involutions of the second kind. Up to the adjoint action of suitable characters of $\Qlat$ and up to $\Aut(\fkg)$-conjugacy, the assignment $(X,\tau) \mapsto \theta(X,\tau)$ defines a map from $\gsat{\aff{A}}$ onto the set of pseudo-involutions of the second kind such that the corresponding restricted Weyl group is a Coxeter group, see \cite[Thm.~2.5, Thm.~4.29]{RV21}.}
		\item \label{rmk:KM:pseudoinvolution:3}
		The map on $\ah^*$ dual to $\theta$ is denoted by the same symbol and preserves $\aQ$. 
		Because $\tsat(c)=-c$, we have $\tsat(\drv{})=-\drv{}$ and hence $\bbZ \drv{}\subseteq \aQ^{-\tsat}$.
		Moreover, $\aQp=\Qp\oplus\bbZ_{\geqslant0}\drv{}$ and thus $\aQp^{-\tsat}=\Qp^{-\tsat}\oplus\bbZ_{\geqslant0}\drv{}$.
		The \emph{restricted rank} 
		of $\tsat$ is the rank of $\aQ^{-\tsat}$, given by the number of $\tau$-orbits in $\wh I \backslash X$, see \eg~\cite[Sec.~4]{RV21} and \cite[Sec.~8.10]{AV22}.
		In particular, $\tsat$ has restricted rank one if and only if $\aQ^{-\tsat}=\bbZ \drv{}$.
		\rmkend
	\end{enumerate}
\end{remarks}


\subsection{Quantum pseudo-involutions}\label{ss:qtheta}

We shall consider a distinguished lift of the pseudo-involution $\tsat$ to an algebra automorphism $\tsatq$ of $\Uqagp$ and $\UqLg$. 
This is obtained by choosing a suitable lift for each of the three factors in $\tsat$. 
First, we consider the standard Chevalley involution on $\Uqag$ given by \eqref{eq:chevalley}.
The diagram automorphism $\tau$ extends canonically to an automorphism of $\Uqagp$ given on the generators by $\tau(\Eg{i})=\Eg{\tau(i)}$, $\tau(\Fg{i})=\Fg{\tau(i)}$, and $\tau(\Kg{h})=\Kg{\tau(h)}$.\\

The action of the Weyl group operator $w_X$ is lifted to $\Uqagp$ as follows. 
Let $\qWS{X}$ be the braid group operator on modules in $\Oint$ corresponding to $w_X$, \cf~\cite[Sec.~5]{Lus94} (see also \cite[Sec.~5]{AV22}).
More precisely, given a reduced expression $s_{i_1}\cdots s_{i_{\ell}}$ of $w_X$ in terms of fundamental reflections, one sets $\qWS{X}\coloneqq \qWS{i_1}\cdots\qWS{i_{\ell}}$, where $\qWS{j}=T''_{j,1}$ 
in the notation from \cite[5.2.1]{Lus94}. 
It follows from the braid relations that $\qWS{X}$ is independent of the 
chosen reduced expression.
We shall consider a Cartan correction of $\qWS{X}$ given by
\begin{align} \label{def:t-Sat}
	\bt{\tsat} \coloneqq \tcorr{\tsat}\cdot\qWS{X}\,,
\end{align}
where $\tcorr{\tsat}$ is the Cartan operator defined as the multiplication by $q^{\iip{\tsat(\lambda)}{\lambda}/2+\iip{\lambda}{\rho_X}}$ on any weight vector of weight $\lambda$, with $\rho_X$ the half-sum of the positive roots of $\ag_X$, see~\cite[Sec.~4.9]{AV22}.
By \cite[Lemma~4.3 (iii)]{AV22}, $\adbt{\tsat}\coloneqq\Ad(\bt{\tsat})$ yields an algebra automorphism of $\Uqagp$.\\

The \emph{quantum pseudo-involution} is the automorphism of $\Uqagp$ given by
\begin{equation} \label{def:thetaq}
	\tsatq\coloneqq\adbt{\tsat}\circ\omega\circ\tau.
\end{equation}
Note that, as in the classical case, the three factors pairwise commute.
We have the basic properties
\begin{equation} \label{thetaq:basics}
	\tsatq|_{U_q(\agp_X)} = \id_{U_q(\agp_X)}, \qq
	\tsatq(K_h) = K_{\theta(h)}, \qq
	\tsatq\big(U_q(\wh\g)_\la \big) = U_q(\wh\g)_{\theta(\la)}\,,
\end{equation}
for any $h\in\aQvext$ and $\lambda\in\aQ$, see, \eg \cite[Lem.~6.10]{AV22}.
Note that $\tsatq$ descends to an automorphism of the quantum loop algebra $\UqLg$.

\begin{remark}
	Further to Remark \ref{rmk:KM:pseudoinvolution}, note that the Kac-Wang extension of $\tau$ to ${\ah}$ does not necessarily preserve the extended coroot lattice $\aQvext\subset{\ah}$, and therefore does not automatically extend to an automorphism of $\Uqag$. 
	To remedy this, one can modify the lattice itself by replacing the standard derivation $\codrv{}\in{\ah}$ with any $\codrv{\tau}\in{\ah}$ such that $\rt{i}(\codrv{\tau})=\rt{\tau(i)}(\codrv{\tau})$, see \cite[Sec.~2.6]{Kol14}. 
	However, for the purposes of this paper, it is sufficient to regard $\tsatq$ as an automorphism of $\Uqagp$.
	In particular, by \cite[Sec.~6.2]{AV22}, ${\ah}^{\tsat}\subset{\ahp}$ and accordingly its natural quantum analogue $U_q({\ah}^{\tsat}) = \langle \Kg{h} \,\vert\, {h\in(\aQvext)^\tsat} \rangle$ embeds in $U_q(\ahp)$.
	\rmkend
\end{remark}


\subsection{QSP subalgebras}\label{ss:qsp} 
Associated to the pseudo-involution $\tsat$, there is a family of coideal subalgebras $\Uqk \subset \Uqagp$, see \cite[Def.~6.11]{AV22}, which is parametrized by two sets $\Parsetc\subset(\bsF^\times)^{\aIS}$ and $\Parsets\subset\bsF^{\aIS}$ which we define below.
The subalgebras thus defined coincide with the subalgebras considered in \cite{Let02,Kol14} up to reparametrization, see \cite[Rmk.~6.13]{AV22}.

\begin{definition}
	The {\em QSP subalgebra} (corresponding to $\tsat$) with parameters $(\Parc,\Pars)$ $\in$ $\Parsetc \times \Parsets$ is the subalgebra
	$\Uqk \subset \Uqagp$ generated by the subalgebras
\begin{equation*}
U_q(\anp_X) = \langle E_i \,\vert\, {i \in X} \rangle, \qq \qq U_q({\ahp}^{\tsat}) = \langle \Kg{h} \,\vert\, {h\in(\aQvext)^\tsat} \rangle
\end{equation*}
and the elements $\Bg{i}$ given by 
	\begin{equation} \label{Bi:def}
		\Bg{i} \coloneqq
		\left\{
		\begin{array}{cl}
			\Fg{i} & \mbox{if }\, i\in X\, ,\\
			\Fg{i} + \parc{i} \, \tsatq(\Fg{i}) + \pars{i} \, \Kg{i}^{-1}  & \mbox{if }\, i\not\in X\, .
		\end{array}
		\right.
	\end{equation}
\end{definition}

The parameter sets $\Parsetc$ and $\Parsets$ are defined as follows, see \cite[Eqns.~(5.9) and (5.11)]{Kol14}. 
First, we fix a subset $\aIS^* \subseteq \aIS \setminus X$ containing a representative for every $\tau$-orbit in $\aIS\setminus X$. 
We set
\begin{align*}
	\aIdiff &\coloneqq \{ i \in \aIS^* \, | \, \tau(i) \ne i \text{ and } \exists j \in X \cup \{ \tau(i)\} \text{ such that } a_{ij} \ne 0 \}, \\
	\aIns &\coloneqq \{ i \in \aIS^* \, | \, \tau(i) =i \text{ and } \, a_{ij} = 0\quad \forall j \in X \}.
\end{align*}
Then, $\Parsetc$ is the set of tuples $\Parc\in(\bsF^\times)^{\aIS}$ such that $\parc{i}=1$ if $i\in X$ and $\parc{i} = \parc{\tau(i)}$ if $\{i , \tau(i) \} \cap \aIdiff = \emptyset$, while $\Parsets$ is the set of tuples $\Pars \in \bsF^{\aIS}$ such that $\pars{i}=0$ if $i\in\aIS\setminus\aIns$ and, for all $(i,j) \in (\aIns)^{\times 2}$, $a_{ij} \in 2\bbZ$ or $\Pars_j = 0$.
Note that $\Parsetc$ and $\Parsets$ do not depend on the choice of $\aIS^*$. 

The constraints given by the parameter sets $\Parsetc$ and $\Parsets$ are motivated by Proposition~\ref{prop:parameter-coideal} below, cf.~\cite[Props.~5.2 and 6.2]{Kol14}.


\begin{prop}\label{prop:parameter-coideal}
	The subalgebra $\Uqk$ is a right coideal in $\Uqagp$ and 
	has minimal intersection with $\Uqahp$, \ie
	\begin{equation*}
		\Delta(\Uqk)\subset\Uqk\ten\Uqagp \qq\mbox{and} \qq \Uqk\cap\Uqahp = U_q(\ahp^{\tsat})\,.
	\end{equation*}
\end{prop}

\begin{remark}\label{rmk:parameter-operator}
	Following \cite[Sec.~7.4]{AV22}, we shall regard the tuple $\Parc$ as a diagonal operator on weight representations. 
	Namely, we fix a group homomorphism $\parc{}:\aP\to\bbF^\times$ such that $\parc{}(\rt{i})\coloneqq\parc{i}$ ($i\in\aIS$). 
	Then, $\parc{}$ acts on any weight vector of weight $\lambda$ as multiplication by $\parc{}(\lambda)$.
	\rmkend
\end{remark}

\begin{example} \label{exam:qOnsager}
	Let $\fkg=\fksl_2$ and $(X,\tau) = (\emptyset,\id)$ (\cf~Example~\ref{exam:sl2GSat}). 
	The corresponding QSP subalgebra is the \emph{q-Onsager algebra} (see, \eg \cite{BK05} and references therein). 
	In our conventions, given parameters $\parc{0},\parc{1} \in \bbF^\times$ and $\pars{0},\pars{1} \in \bbF$, it is the coideal subalgebra of $U_q(\wh\fksl_2)$ generated by the elements
	\begin{equation}\label{qOnsager:generators}
		B_i \coloneqq F_i - q^{-1} \parc{i} \, E_i\,\Kg{i}^{-1}  + \pars{i} \, \Kg{i}^{-1}
	\end{equation}
	with $i=0,1$.
	Note that the parameters $\pars{i}$ are allowed to be nonzero, since in this case $\aIns = \{0,1\}$.
	\rmkend
\end{example}


\subsection{Gauge transformations and cylindrical structures}\label{ss:univ-kmx}

In \cite[Thms.~8.8-8.9]{AV22}, we proved that the QSP subalgebra $\Uqk$ gives rise to a family of K-matrices in $\Uqag$.
The result crucially relies on the construction of the \emph{quasi-K-matrix}, due to Bao and Wang \cite{BW18} and generalized in \cite{BK19,AV22}.
Let $\Gg$ be the group of invertible elements $\gau\in \CUqag{}$ such that $\Ad(\gau)$ preserves $\Uqag\subset\CUqag{}$. 
We fix $\gau \in \Gg$ and set
\begin{equation}\label{twistmapandDrinfeldtwist:formula}
\psi \coloneqq \Ad(\gau) \thinspace \circ \thinspace \tsatq^{-1} \qq \mbox{and}\qq
J \coloneqq (\gau\ten\gau) \cdot \RM{\theta} \cdot \Delta(\gau)^{-1}
\end{equation}
where $\RM{\tsat} = (\bt{\tsat}\ten\bt{\tsat})^{-1}\cdot \Delta(\bt{\tsat})$ is the R-matrix of $U_q(\ah^\theta)U_q(\ag_X)$.

\begin{theorem} \cite[Thms.~8.8-8.9]{AV22} \label{thm:av-k-mx}
	There exists a unique series 
	\begin{equation*}
	\QK{}= \sum_{\mu \in (\aQp)^{-\tsat}}\QK{\mu} \qq \mbox{with} \qq \QK{\mu}\in\Uqanp_\mu
	\end{equation*}
	such that $\QK{0}=1$ and, for any $\gau\in\Gg$, the operator
	\begin{equation} \label{eq:universal-K-matrix}
	\KM{}=\gau\cdot\Parc^{-1}\cdot\QK{}\in\CUqag{}
	\end{equation}
	satisfies the intertwining identity
	\begin{align}\label{eq:k-intertwiner}
		\KM{}\cdot b=\psi(b)\cdot\KM{}  \qquad(b \in \Uqk)
	\end{align}
	and the coproduct identity
	\begin{align}\label{eq:coprod-id}
		\Delta(\KM{}) = J^{-1} \cdot (1\ten\KM{}) \cdot \RM{}^{\psi} \cdot (\KM{}\ten 1).
	\end{align}
	Moreover, the following (parameter-independent) generalized reflection equation holds:
	\begin{align}\label{eq:tw-RE}
	\RM{21}^{\psi\psi} \cdot (1\ten\KM{}) \cdot \RM{}^{\psi} \cdot (\KM{}\ten 1)
	= (\KM{}\ten 1) \cdot (\RM{}^{\psi})_{21} \cdot (1\ten\KM{}) \cdot \RM{}.
\end{align}
\end{theorem}

The evaluation of $\KM{}$ on $V\in\Oint$ yields a QSP intertwiner 
\begin{equation*}
	\KM{V}:V\to\twistrep{V}{\psi}
\end{equation*} 
where $\twistrep{V}{\psi}$ denotes the pullback of $V$ via $\psi$.
We shall refer to the automorphism $\psi$ as the \emph{twisting operator} of the reflection equation.

\begin{remark}
		We call the operator $\QK{}$ {the quasi-K-matrix} of $\Uqk$, although this terminology is used in \cite{BW18, BK19} for the operator ${\mathfrak{X}}=\ol{\QK{}}$. 
		In \cite{AV22}, we provide a more general construction of the operator $\QK{}$, which does not rely on the existence of the so-called bar involution for $\Uqk$, see \cite{BK15, BW18, ES18, BK19}. 
		In fact, $\QK{}$ can be used to \emph{define} this involution, as later observed in \cite{Kol22}.
\rmkend
\end{remark}

Theorem~\ref{thm:av-k-mx} shows that the quasi-K-matrix $\QK{}$ is itself a universal solution of \emph{a} generalized reflection equation and that a family of compatible universal K-matrices, twisting operators and Drinfeld twists can be produced by simultaneously ``gauge-transforming'' $\psi$, $J$ and $K$ in terms of a suitable element $\gau$. 
The choice of gauge is relevant for applications (\eg \cite{AP22}), since the generalized reflection equation \eqref{eq:tw-RE} depends on $\psi$.

\begin{example}\label{rmk:distinguished}
	We describe some distinguished examples of $\gau$.
	
	\begin{enumerate}\itemsep0.25cm
		\item \label{rmk:distinguished:1} 
		Gauging by $\gau={\bm \beta}$, where ${\bm \beta}$ is any map $\aP\to\bsF^{\times}$, 
		yields \emph{diagonal} corrections.
		In this case $K$ acts on modules in $\Oinf$. 
		In particular, if $\gau = 1$, the resulting universal K-matrix, called \emph{standard K-matrix}, satisfies the intertwining identity 
		$K \cdot b = \theta_q^{-1}(b) \cdot K$ for all $b \in \Uqk$.
		
		\item \label{rmk:distinguished:2} 
		For $\gau=\bt{\tsat}$, the twisting operator reduces to the involutive bialgebra isomorphism $\psi=\chev\circ\tau: \Uqag \to \Uqag^{\sf cop}$ 
		and the Drinfeld twist $J = 1 \ten 1$ is trivial. 
		The corresponding K-matrix, which we refer to as the \emph{semi-standard K-matrix}, 
		was also considered (up to conventions) in \cite[Thm.~7.5]{BK19} and \cite[Thm.~3.15]{BW21}

		\item \label{rmk:distinguished:3} 
		More generally, let $(Y,\eta)\in\gsat{\aff{A}}$ be a generalized Satake diagram and $\zsat$ the corresponding pseudo-involution. 
		For $\gau=\bt{\zsat}^{-1}\bt{\tsat}$, one recovers the combinatorial family of universal K-matrices constructed in \cite{AV22}.
		Among these, the standard K-matrix corresponds to the choice $\zsat=\tsat$, i.e. $(Y,\eta) = (X,\tau)$, 
		and the semi-standard K-matrix to the case $\zsat = \om$ with $(Y,\eta) = (\emptyset, \id)$.
		The finite-type K-matrix constructed by Balagovi\'{c} and Kolb in \cite{BK19} corresponds instead to the choice $\zsat = \id$, 
		given by $(Y,\eta) = (\fIS,\oi_{\fIS})$, which is available only in finite type. 
		In analogy with this last choice, in Section~\ref{s:recover-RE}, we will discuss the choice with $Y=\aIS\setminus\{0,\tau(0)\}$ and 
		$\eta \in \Aut(\aff{A})$ is defined by $\eta(0)=\tau(0)$, and $\eta|_{Y}=\oi_{Y}$. \hfill\rmkend
	\end{enumerate}
\end{example}

The coproduct identity \eqref{eq:coprod-id} and the generalized reflection equation \eqref{eq:tw-RE} admit a similar representation-theoretic interpretation. 
First, note that, for any weight $\Uqag$-module $M$ and $W\in\Oinf$, the operator $\RM{MW}$ is always well-defined, whilst this is not true for $\RM{WM}$.
Recall that any algebra automorphism of $\Uqag$ preserves $\Uqah$ and thus its pullback preserves weight modules\ourcomment{It is enough to observe that the only invertible elements are in $\Uqah$. More precisely, for any algebra automorphism $\phi$ we must have $\phi(K_h)=\pm K_{\Phi(h)}$ for some group automorphism $\Phi:\aQvext\to\aQvext$.  Note however that type $\bf 1$ representations are not necessarily preserved.}.
Therefore, $\twistrep{V}{\psi}$ is still a weight module and $\RM{\twistrep{V}{\psi}W}$ is well-defined.
Moreover, we have $\RM{21}^{\psi\psi} = J_{21} \cdot \RM{}\cdot J^{-1}$\,. Thus, the operator $\RM{\twistrep{V}{\psi}\twistrep{W}{\psi}}$ is also well-defined and the two sides of \eqref{eq:tw-RE} act naturally on tensor products in $\Oint$.

\begin{remark}
In terms of braid group actions, the generalized reflection equation yields an action of a cylindrical ribbon braid {\em groupoid}, since as mentioned above the operators involved cannot be arbitrarily composed. 
Unlike the situation for the quantum group $U_q(\fkg)$, pulling back by a twisting operator $\psi$ of the form given by \eqref{twistmapandDrinfeldtwist:formula} does not preserve the category $\Oint$.
However, for finite-dimensional $\UqLg$-modules, in special cases, this action does yield a representation of the cylindrical braid group, \eg if $\twistrep{V}{\psi}=V$ and $\twistrep{W}{\psi}=W$ (see~Section~\ref{s:recover-RE}). \rmkend
\end{remark}


\section{Spectral K-matrices}\label{s:spectral}

In this section, we present the first main result of the paper.
Let $(X,\tau)\in\gsat{\aff{A}}$ with pseudo-involution $\tsat$, $(\parc{},\pars{})\in\Parsetc\times\Parsets$, and $\Uqk\subset\Uqagp$ the  corresponding QSP subalgebra. 
We prove that, under mild assumptions, the universal K-matrices constructed in \cite{AV22} and in Theorem~\ref{thm:av-k-mx} specialize to spectral operators on finite-dimensional $\UqLg$-modules. 

\subsection{$\tau$-minimal grading shifts}\label{ss:tau-grading}
We shall need to replace the homogeneous grading shift \eqref{eq:grading-shift}, most commonly used in the context of 
quantum loop algebras, with a distinguished 
$\tau$-invariant grading shift. Note that for any group homomorphism $s: \aQ \to \Z$, there is 
a grading shift $\Sigma^s_z:\UqLg\to\UqLg[z,z^{-1}]$ 
\ourcomment{We assume clear that $\shift{z}$ is then 
	extended by linearity to $\UqLg[z,z^{-1}]$. }
defined by 
\begin{equation*}
	\Sigma^s_z(E_i) = z^{s(\al_i)} E_i, \qq \Sigma^s_z(F_i) = z^{-s(\al_i)} F_i, \qq \Sigma^s_z(K_h) = K_h
\end{equation*}
for $i \in \wh I$, $h \in \aQv$. Then, $\Sigma^s_z$ is \emph{$\tau$-invariant} if $s \circ \tau = s$ and therefore $\Sigma^s_z \circ \tau = \tau \circ \Sigma^s_z$. Note that $\Sigma^s_z$ is $\tau$-invariant if and only if, as a function on the set of affine simple roots, $s$ is the characteristic function of a union of $\tau$-orbits. For instance, the principal grading shift $\shift{z}^{\sf pr}$ corresponding to $s(\rt{i})=1$ ($i\in\aIS$) is always $\tau$-invariant.
The \emph{$\tau$-minimal} grading shift $\Sigma^{s_\tau}_z$ corresponds instead to the characteristic function $s_\tau$ of the $\tau$-orbit of the affine node $0$, \ie
\begin{equation*}
	s_\tau(\al_i) = \begin{cases}
		1 & \text{if } i \in \{ 0,\tau(0) \}, \\
		0 & \text{otherwise}\,.
	\end{cases}
\end{equation*}
Note that $\shift{z}^{s_\tau} = \shift{z}^{\sf hom}$ if and only if $\tau(0)=0$.
Moreover, the analogue of Theorem~\ref{thm:spectral-R} holds if we replace $R(z)$ by $(\id \ten \Sigma^{s_\tau}_z)(R)$ and $\Delta_z$ by $(\id\ten\Sigma^{s_\tau}_z) \circ \Delta$.\\

Henceforth, we shall only use the $\tau$-minimal grading shift $\Sigma_z^{s_\tau}$ and drop the upper index $s_{\tau}$ unless needed. 
In particular, we shall denote by $\pi_{V,z}$ the action on $V$ shifted by $\Sigma_z^{s_{\tau}}$.

\subsection{Spectral K-matrices}\label{ss:spectral-k}

Let $\GQSP\subset\cG$ be the subset of gauge transformations $\gau\in\cG$ 
of the form $\gau\coloneqq S_Y^{-1}S_X {\bm \beta}^{-1}$, where 

\begin{enumerate}\itemsep0.25cm
	\item $Y\subset\aff{I}$ is any proper subdiagram such that $s(\al_i)=0$ for any $i\in Y$;
	\item ${\bm \beta}:\aP\to\bsF^\times$ is any map such that $\parc{}(\drv{}){\bm \beta}(\drv{})=1$.
\end{enumerate}

The definition of $\GQSP$ is motivated by the results of the next section, 
see~Remark~\ref{rmk:rational:2}.
In Section~\ref{ss:univ-kmx}, we considered the twisting operators of the form 
\begin{equation}\label{psi:QSPadmissible}
	\psi=\Ad(\gau)\circ\tsat_q^{-1}
\end{equation}
where $\tsat_q$ is defined by \eqref{def:thetaq} and $\gau\in\cG$. 
The twisting operator $\psi$ is  {\em QSP-admissible} whenever $\gau\in\GQSP$. 

Henceforth, we assume that $\psi$ is a QSP-admissible twisting operator.
We have the following spectral analogue of Theorem~\ref{thm:av-k-mx}.

\begin{theorem}\label{thm:spectral-k}
	The quantum loop algebra $\UqLg$ has a $\GQSP$-family of universal \emph{spectral} K-matrices relative to the QSP subalgebra $\Uqk$. 
	More precisely, for any $\gau\in\GQSP$, set
	\begin{equation*}
		\psi \coloneqq \Ad(\gau) \circ \tsat_{q}^{-1} \, , \qq\mbox{and} \qq 
		J \coloneqq (\gau\ten\gau) \cdot \RM{\theta} \cdot \Delta(\gau)^{-1}.
	\end{equation*}
	There is a canonical Laurent series $\sKM{}{z}\in\Lfml{\CUqLg{}}{z}$
	such that $\shift{a}(\sKM{}{z})=\sKM{}{az}$ ($a\in\bsF^\times$) and the following properties hold.
	\vspace{0.25cm}
	\begin{enumerate}\itemsep0.25cm
		\item For any $b\in\Uqk$, 
		\begin{align}\label{eq:spectral-k-intertwiner}
			\sKM{}{z}\cdot\shift{z}(b)=\psi(\shift{1/z}(b))\cdot\sKM{}{z}\, .
		\end{align}
		\item 
		Set $\sRM{}{z}^{\psi} \coloneqq (\psi\ten \id)(\sRM{}{z}) $. 
		Then,
		\begin{align}\label{eq:spectral-k-coproduct}
			\Delta_{w/z}(\sKM{}{z})=J^{-1}\cdot (1\ten\sKM{}{w}) \cdot\sRM{}{zw}^{\psi} \cdot (\sKM{}{z} \ten 1)\, .
		\end{align}
	\end{enumerate}
	\vspace{0.25cm}
	Moreover, $\sKM{}{z}$ is a solution of the generalized reflection equation 
	\begin{equation}\label{eq:spectral-tw-RE}
	\begin{split}
	\sRM{}{w/z}_{21}^{\psi\psi} \cdot (1\ten\sKM{}{w})& \cdot \sRM{}{zw}^{\psi} \cdot (\sKM{}{z}\ten 1)=\\
	&= (\sKM{}{z}\ten 1) \cdot \sRM{}{zw}_{21}^{\psi} \cdot (1\ten\sKM{}{w}) \cdot \sRM{}{w/z}
	\end{split}
	\end{equation}
	with $\sRM{}{z}_{21}^{\psi} \coloneqq (\psi\ten\id)(\sRM{}{z})_{21}$.
\end{theorem}

\begin{remark}\label{rmk:spectral-k}\hfill
	\begin{enumerate}[leftmargin=2em]\itemsep0.25cm
		\item \label{rmk:spectral-k:1}
		The identities \eqref{eq:spectral-k-coproduct} and \eqref{eq:spectral-tw-RE} hold in $\Lfml{\CUqLg{2}}{w/z,z}$, 
		where $\Lfml{\bsF}{w/z,z}\coloneqq\operatorname{Frac}(\fml{\bsF}{w/z,z})$.
		Following \cite[Eq.~(10)]{Che84}, it may be convenient to use an {adapted} set of coordinates, given by $u=w/z$ and $v=z$. 
		Then, \eqref{eq:spectral-k-intertwiner} and \eqref{eq:spectral-k-coproduct} read
		\begin{align*}
			\Delta_{u}(\sKM{}{v})=J^{-1}\cdot (1\ten\sKM{}{uv}) \cdot \sRM{}{uv^2}^{\psi} \cdot (\sKM{}{v}\ten 1) \, 
		\end{align*}
		and
		\begin{align*}
			\sRM{}{u}_{21}^{\psi\psi} \cdot (1\ten\sKM{}{v}) \cdot& \sRM{}{uv^2}^{\psi} \cdot (\sKM{}{uv}\ten 1) = \\
			&=
			(\sKM{}{uv} \ten 1) \cdot \sRM{}{uv^2}_{21}^{\psi} \cdot (1\ten\sKM{}{v}) \cdot\sRM{}{u}\, 
		\end{align*}
		in $\Lfml{\CUqLg{\ten 2}}{u,v}$.
		\item \label{rmk:spectral-k:2}
		By definition, every element in $\GQSP$ has the form $\gau=S_Y^{-1}S_X{\bm\beta}^{-1}$ for some $Y\subsetneq\aIS$ such that $\shift{z}(S_Y)=S_Y$ and ${\bm\beta}:\Plat\to\bsF^{\times}$  (\cf~Section~\ref{ss:spectral-k}).
		Therefore $\gau$ is shift-invariant if and only if $S_X$ is shift-invariant.
		In this case, the K-matrix 
		$\sKM{}{z}$ is a formal series in $\fml{\CUqLg{}}{z}$.
		\item \label{rmk:spectral-k:3}
		The QSP subalgebra of $U_q(L\fksl_2)$ with $(X,\tau)=(\emptyset,(01))$, see Example \ref{exam:sl2GSat}, is known as the \emph{augmented q-Onsager algebra}. 
		In \cite[Sec.~4.1.2]{BT18}, so-called \emph{generic K-operators} in a completion of $U_q(\wh\fkh)(z)$ are considered.
		It would be interesting to connect these to the spectral K-matrices considered here.
		\rmkend
	\end{enumerate}
\end{remark}

In analogy with the case of the R-matrix, the spectral K-matrix is obtained by applying the shift operator to the universal K-matrix $\KM{}$ from Theorem~\ref{thm:av-k-mx}, \ie $\sKM{}{z}\coloneqq\shift{z}(\KM{})$.
Namely, the identity \eqref{eq:spectral-k-intertwiner} is recovered from \eqref{eq:k-intertwiner} by applying the shift operator $\shift{z}$.
Similarly, the identities \eqref{eq:spectral-k-coproduct}, and \eqref{eq:spectral-tw-RE} are recovered from \eqref{eq:coprod-id} and \eqref{eq:tw-RE}, respectively, by applying the shift operator $\shift{z}\ten\shift{w}$. 
Clearly, since the operator $\sKM{}{z}$ is valued in $\CUqLg{}$, the statements above are to be interpreted as operators on finite-dimensional modules in $\ModfdUqLg$.
Therefore, it is necessary to prove that $\sKM{}{z}$ gives rise to a well-defined element $\sKM{V}{z}\in\Lfml{\End(V)}{z}$ for any $V\in\ModfdUqLg$. 
The proof of Theorem~\ref{thm:spectral-k} is carried out in Sections~\ref{ss:k-fd-rep}-\ref{ss:k-fd-rep-2}.

\subsection{Descent to finite-dimensional modules}\label{ss:k-fd-rep}
The first step in the proof of Theorem~\ref{thm:spectral-k} amounts to proving that, 
for any $V\in\ModfdUqLg$ with action $\pi_V:\UqLg\to\End(V)$, we obtain a 
well-defined operator 
\begin{equation}
	\sKM{V}{z}\coloneqq\pi_{V,z}(\KM{})=\pi_{V}\circ\shift{z}(\KM{})
	\in\Lfml{\End(V)}{z}\,.
\end{equation}
More precisely, we shall prove that each of the operators involved in the definition of the universal K-matrix \eqref{eq:universal-K-matrix} and in the coproduct identity \eqref{eq:spectral-k-coproduct} descends to one on any finite-dimensional $\UqLg$-module.

\begin{prop}\label{prop:descent}
	The following operators in $\CUqag{}$ descend to $\CUqLg{}$:
	\vspace{0.25cm}
	\begin{enumerate}\itemsep0.25cm
		\item 
		the operator $\tcorr{\tsat}$, defined on any weight vector $v$ of weight $\lambda$ by
		\begin{equation}
			\tcorr{\tsat}(v)=q^{\iip{\tsat(\lambda)}{\lambda}/2+\iip{\lambda}{\rho_X}}\cdot v\, ,
		\end{equation}
		where $\rho_X$ is the half-sum of the positive roots of $\ag_X$ (\cf~ Section~\ref{ss:qtheta});

		\item 
		for any $i\in\aIS$, the braid group operator $\qWS{i}$;

		\item 
		for any map $\bm{\phi}:\aP\to\bbF^\times$ such that $\bm{\phi}(\drv{})=1$, 
		the diagonal operator defined on any weight vector $v$ of weight $\lambda$ by $\bm{\phi}(v)\coloneqq\bm{\phi}(\lambda)\cdot v$.
	\end{enumerate}	
	Moreover, 
	\begin{enumerate}
		\item[(d)] 
		for any $\Psi\in\wh{\bigoplus}_{\mu\in\aQp}\Uqanp_{\mu}$, $\shift{z}(\Psi)$ descends to $\fml{\CUqLg{}}{z}$.
	\end{enumerate}
\end{prop}

\begin{proof} \hfill
\begin{enumerate}[leftmargin=2em]\itemsep0.25cm
	\item
	In analogy with the operator $q^{\wh{\Omega}_0}$ from Section~\ref{ss:R-matrix}, one checks easily that 
	\begin{align}
		\tcorr{\tsat}=q^{\sum_{i\in\fIS} \tsat(u^i)u_i+m(\tsat(c)d+\tsat(d)c)-\rho^\vee_{X}}\, ,
	\end{align}
	where $m=1,2,3$ if $\g$ is of type $\mathsf{ADE}$, $\mathsf{BCF}$, or $\mathsf{G}$, respectively. 
	Since $\tsat(c)=-c$, the term $m(\tsat(c)d+\tsat(d)c) = m(\tsat(d)-d)c$ acts as 0 and can be ignored.
	Therefore, as in the case of the operator $q^{\wh{\Omega}_0}$ in Section~\ref{ss:spectral-R}, $\tcorr{\tsat}$
	descends to an operator in $\CUqLg{}$. Note that $\tcorr{\tsat}(\drv{})=1$.	
	
	\item
	By restriction, $V$ is a finite-dimensional representation of $U_q(\wh\g)_{\{i\}}$ and therefore integrable.
	In particular, the action of $\qWS{i}$ on $V$ is well-defined.

	\item	
	By definition, any type $\bf{1}$ $\UqLg$-module $V$ admits a weight 
	decomposition over the quotient lattice $\aPd$, \ie $V=\bigoplus_{\lambda\in\aPd} \wsp{V}{\lambda}$.
	Therefore, the operator $\bm{\phi}$ acts on $V$ if and only if 
	it factors through $\aPd$, \ie if $\bm{\phi}(\drv{})=1$.
	
	\item 
	Set $\Psi=\sum_{\mu\in\aQp}\Psi_{\mu}$, with $\Psi_{\mu}\in\Uqanp_{\mu}$.
	Then, $\shift{z}(\Psi)=\sum_{n\geqslant0}\Psi(n)z^n$ where $\Psi(n)\coloneqq\sum_{\mu\in s^{-1}(n)}\Psi_{\mu}$. 
	Fix $n\geqslant0$. 
	We shall prove\footnote{
		Note that, in the case of the principal grading shift (\cf~Section~\ref{ss:tau-grading}), $s^{-1}(n)$ is a finite set and the result is clear.
	} that $\Psi(n)$ is a well-defined operator on $V$. 
	There exists $i \in \aIS$ such that $s(\al_i) \neq 0$.
	Any $\mu\in s^{-1}(n)$ has the form $\mu=m\rt{i}+\lambda$, where $m\leqslant n$ and $\lambda\in {\sf Sp}_{\Z \ge 0} \{ \al_j \, | \, j \in \aIS \backslash \{ i \} \} \subset \aQp$. 
	However, the action of $\langle E_j \,\vert\, j \in \aIS \backslash \{i\} \rangle$ is locally finite on $V$ and the number of occurrences of $\Eg{i}$ is bounded by $n$. 
	Thus, the action of $\Psi_{\mu}$ on $V$ is nonzero only for finitely many $\mu\in s^{-1}(n)$ and $\Psi(n)$ is well-defined on $V$. \hfill \qedhere
\end{enumerate}
\end{proof}

Recall that, by Theorem~\ref{thm:av-k-mx} and \eqref{eq:universal-K-matrix}, we have $\KM{}=\gau \cdot\Parc^{-1}\cdot\QK{}$,
where $\gau\in\GQSP$ is a gauge, $\Parc{}$ is  the parameter operator, and 
$\QK{}$ is the quasi-K-matrix. 
Note that, by definition of $\GQSP$, $\shift{z}(\gau\cdot\Parc{}^{-1})$ is an element of $\CUqLg{}[z,z^{-1}]$.
Therefore, by the result above, $\sKM{}{z}\coloneqq\shift{z}(\KM{})\in\Lfml{\CUqLg{}}{z}$, \ie for any $V\in\ModfdUqLg$, $\sKM{V}{z}\coloneqq \pi_{V,z}(\KM{})$ is a well-defined operator in $\Lfml{\End(V)}{z}$.

\begin{remark}\label{rmk:special-formal-cases}
We have not yet used the fact that the grading shift is $\tau$-invariant and $\gau\cdot\bt{\tsat}$ is shift-invariant. 
These assumptions become crucial in the next step.
	\rmkend
\end{remark}


\subsection{Spectral K-matrices on finite-dimensional modules}\label{ss:k-fd-rep-2}
We complete the proof of Theorem~\ref{thm:spectral-k} by showing that the relations
\eqref{eq:spectral-k-intertwiner}, \eqref{eq:spectral-k-coproduct}, and \eqref{eq:spectral-tw-RE} hold.
Note that, since the grading shift is $\tau$-invariant and $\gau\in\GQSP$ is shift-invariant, 
we have $\tau\circ\shift{z}=\shift{z}\circ\tau$ and $\Ad(\gau\cdot\bt{\tsat})\circ\shift{z}=\shift{z}\circ\Ad(\gau\cdot\bt{\tsat})$.
Finally, since $\chev\circ\shift{z}=\shift{1/z}\circ\chev$ and $\psi=\Ad(\gau\cdot\bt{\tsat})\circ\chev\circ\tau$, we have
\begin{align}\label{eq:parameter-inversion}
	\psi\circ\shift{z}=\shift{1/z}\circ\psi\, .
\end{align}
Let $V\in\ModfdUqLg$. 
The action on the shifted representation $\Lfml{\twistrep{V}{\psi}}{1/z}$ is, by definition, given by
\begin{equation}\label{eq:pullback}
	\pi_{\twistrep{V}{\psi}, 1/z}(x)=\pi_{\twistrep{V}{\psi}}(\shift{1/z}(x))=\pi_V(\psi\circ\shift{1/z}(x))\, .
\end{equation}
Since $\pi_V(\psi\circ\shift{1/z}(x))=\pi_V(\shift{z}\circ\psi(x))$, one has
$\Lfml{\twistrep{V}{\psi}}{1/z}=\twistrep{\Lfml{V}{z}}{\psi}$.
Therefore, the evaluation of the intertwining identity \eqref{eq:k-intertwiner}
on $\Lfml{V}{z}$ through $\pi_{V,z}$ yields
\begin{align}\label{eq:spectral-k-intertwining-rep}
	\sKM{V}{z}\pi_{V,z}(b)=\pi_{\twistrep{V}{\psi}, 1/z}(b)\sKM{V}{z}\, .
\end{align}
It follows that $\sKM{V}{z}$ is a QSP intertwiner $\Lfml{V}{z}\to\Lfml{\twistrep{V}{\psi}}{1/z}$,
which is equivalent to \eqref{eq:spectral-k-intertwiner}.\\

Set $\sKM{VW}{z,w} \coloneqq (\pi_{V,z}\ten\pi_{W,w})(\Delta(\KM{}))$.
Then, the evaluation of the coproduct identity \eqref{eq:coprod-id} on $\Lfml{V}{z}\ten\Lfml{W}{w}$ through $\pi_{V,z}\ten\pi_{W,w}$ yields
\begin{align}\label{eq:spectral-k-coproduct-rep}
	\sKM{VW}{z,w} = J_{VW}^{-1}\cdot (\id\ten\sKM{W}{w}) \cdot \sRM{\twistrep{V}{\psi}W}{zw} \cdot (\sKM{ V}{z}\ten \id)\, ,
\end{align}
By \eqref{eq:pullback}, this is equivalent to \eqref{eq:spectral-k-coproduct}.
Finally, the evaluation of the generalized reflection equation \eqref{eq:tw-RE} on $\Lfml{V}{z}\ten\Lfml{W}{w}$ through $\pi_{V,z}\ten\pi_{W,w}$
yields
\begin{align}\label{eq:spectral-tw-RE-rep}
	& \sRM{W^{\psi}\, V^\psi}{w/z}_{21} \cdot (\id\ten\sKM{ W}{w}) \cdot \sRM{\twistrep{V}{\psi}W}{zw} \cdot (\sKM{ V}{z}\ten \id) = \\
	& \qquad \qquad = (\sKM{ V}{z} \ten \id) \cdot \sRM{\twistrep{W}{\psi}\, V}{zw}_{21} \cdot (\id\ten\sKM{W}{w}) 	\cdot \sRM{VW}{w/z}\, ,
\end{align}
where $\sRM{WV}{z}_{21}\coloneqq(1\, 2)\circ\sRM{WV}{z}\circ (1\, 2)$.
As before, this is equivalent to \eqref{eq:spectral-tw-RE}.

\section{Trigonometric K-matrices}\label{s:rational}
\summary{In this section,
	\begin{itemize}
		\item \ref{ss:rational-k}: trigonometric K-matrices (thm)
		\item \ref{ss:proof-rational-1}: proof of theorem part 1
		\item \ref{ss:proof-rational-2}: proof of theorem part 2
		\item \ref{ss:unitarity}: unitary K-matrices  for $\twistrep{V}{\psi^2}=V$
		\item \ref{ss:unitarity-bis}: unitary K-matrices for $\twistrep{V}{\psi}=V$
		\item \ref{ss:abelian-k-mx}: abelian K-matrices
	\end{itemize}
}

In this section, we prove that the spectral K-matrix constructed in Theorem~\ref{thm:spectral-k} gives rise to trigonometric solutions 
(\ie operator-valued rational functions in $z$) of the generalized reflection equation on irreducible $\UqLg$-modules. 

\subsection{Trigonometric K-matrices}\label{ss:rational-k}
\label{ss:rational-intertwiners}
The construction of the spectral K-matrix from Theorem~\ref{thm:spectral-k} immediately implies the existence of a trigonometric QSP intertwiner on any finite-dimensional $\UqLg$-module.\footnote{
	We are grateful to V. Toledano Laredo for pointing this argument out to us.
}

\begin{lemma}\label{lem:rational-intertwiner}
	Let $V\in\ModfdUqLg$. 
	There exists a QSP intertwiner $\rKM{V}{z}:\shrep{V}{z}\to\shrep{\twistrep{V}{\psi}}{1/z}$
	in $\End(V)(z)$.
\end{lemma}

\begin{proof}
	By Theorem~\ref{thm:spectral-k} and Section~\ref{ss:k-fd-rep-2}, for any $V\in\ModfdUqLg$, 
	the element $\sKM{}{z}$ provides an intertwiner 
	\begin{equation}
		\sKM{V}{z}:\Lfml{V}{z}\to\Lfml{\twistrep{V}{\psi}}{1/z}\, . 
	\end{equation}
	This is equivalent to the existence of a solution $\sKM{ V}{z}\in\Lfml{\End(V)}{z}$
	of a finite system of linear equations, namely
	\begin{equation}\label{eq:spectral-k-intertwiner-V}
		\sKM{V}{z}\cdot\pi_{V,z}(b)=\pi_{\twistrep{V}{\psi}, 1/z}(b)\cdot\sKM{V}{z}\, ,
	\end{equation}
	where $b\in\Uqk$ runs over any finite set of generators of $\Uqk$. 
	Since the system is consistent and defined over $\bsF(z)$, it admits a solution in $\End(V)(z)$.
\end{proof}

It follows that, in the case of irreducible representations, 
the spectral K-matrix $\sKM{V}{z}$ is rational up to a nonzero scalar in $\Lfml{\bsF}{z}$.

\begin{theorem}\label{thm:rational-k}
	Let $V,W \in \ModfdUqLg$ be irreducible modules. 
	\begin{enumerate}\itemsep0.25cm
		\item \label{thm:rational-k:1D}
		The space of QSP intertwiners $\Lfml{V}{z}\to\Lfml{\twistrep{V}{\psi}}{1/z}$ is one-dimensional.
		\item \label{thm:rational-k:factorization}
		There exist a formal Laurent series $g_{V}(z)\in\Lfml{\bsF}{z}$ and a non-vanishing operator-valued polynomial $\rKM{V}{z}\in\End(V)[z]$, uniquely defined up to a scalar in $\bsF^\times$, such that
		\begin{align}\label{eq:rational-decomp-k}
			\sKM{V}{z}= g_{V}(z)\cdot\rKM{V}{z}\, .
		\end{align}
		\ourcomment{Possibly, one can normalize by setting $z^{{\sf ord}(0)}g(z)|_{z=0}=1$.}
		\item \label{thm:rational-k:RE}
		The operators $\rKM{V}{z}$ and $\rKM{W}{w}$ satisfy the generalized reflection equation in $\End(V\ten W)(z,w)$ 
		\begin{align}\label{eq:rational-tw-RE}
			& \rRM{\twistrep{W}{\psi}\, \twistrep{V}{\psi}}{\tfrac{w}{z}}_{21} \cdot (1\ten\rKM{ W}{w}) 
				\cdot \rRM{\twistrep{V}{\psi}W}{zw}\cdot (\rKM{V}{z}\ten 1) = \\
			& \qquad \qquad = (\rKM{V}{z}\ten 1) \cdot \rRM{\twistrep{W}{\psi}\, V}{zw}_{21} 
				\cdot (1\ten\rKM{W}{w}) \cdot \rRM{VW}{\tfrac{w}{z}}\, ,
		\end{align}
		where $\rRM{VW}{z}$ is the trigonometric R-matrix (see Section~\ref{ss:rational-R}), and
		\begin{equation*}
			\rRM{WV}{z}_{21}\coloneqq(1\, 2)\circ\rRM{WV}{z}\circ (1\, 2).
		\end{equation*}
	\end{enumerate}
	We shall refer to the operator $\rKM{V}{z}$ as a {\em polynomial trigonometric K-matrix}.
\end{theorem}

The proof of Theorem~\ref{thm:rational-k} is carried out in Sections~\ref{ss:proof-rational-1}-\ref{ss:proof-rational-2}.
Note that, from Remark~\ref{rmk:spectral-k} \ref{rmk:spectral-k:2}, we get the following.

\begin{corollary} \label{cor:rational-regular-k}
	If $\shift{z}(\gau)=\gau$, the universal K-matrix descends on any irreducible $\UqLg$-module $V$ to a formal series operator $\sKM{ V}{z}\in\fml{\End(V)}{z}$, endowed with a unique factorization
	\begin{align*}
		\sKM{V}{z}=g_V(z)\cdot\rKM{V}{z}\,,
	\end{align*}
	where $g_V(z)\in\fml{\bsF}{z}$ is such that $g_V(0)=1$ and $\rKM{V}{z}\in\End(V)[z]$ is non-vanishing.
\end{corollary}

\subsection{Proof of Theorem~\ref{thm:rational-k}, part \ref{thm:rational-k:1D}}\label{ss:proof-rational-1}
We say that an irreducible $\UqLg$-module $V$ is \emph{generically QSP irreducible} 
if $\Lfml{V}{z}$ is irreducible as a representation over $\Lfml{\Uqk}{z}$. 
Such a condition is the natural counterpart of the generic irreducibility of the tensor product $V\ten \Lfml{W}{z}$ in Theorem \ref{thm:rational-R}, which holds for any pair of irreducible representations $V,W$.\\

The condition of generic irreducibility may depend on the choice of a grading shift. 
In Corollary~\ref{cor:QSP:irreducible}, we prove that every irreducible $\UqLg$-representation is 
generically QSP irreducible with respect to the {\em principal} grading shift. 
It is then clear that, in this case, \ref{thm:rational-k:1D} follows by Schur's lemma. 
Namely, let $\cK_1,\cK_2\in\Lfml{\End(V)}{z}$ be two  solutions of \eqref{eq:spectral-k-intertwiner-V}.
Let 
\begin{equation*}
	\Pfml{\bsF}{z}\coloneqq \bigcup_{n>0}\Lfml{\bsF}{z^{1/n}}
\end{equation*}
be the field of Puiseux series over $\bsF$, \ie the algebraic closure of $\Lfml{\bsF}{z}$, and set
$\Pshrep{V}{z}\coloneqq V\ten\Pfml{\bsF}{z}$. 
The composition $\cK_2^{-1}\cK_1:\Pfml{V}{z}\to\Pfml{V}{z}$ is an intertwiner.
Therefore, by generic QSP irreducibility and Schur's lemma, there exists $g(z)\in\Pfml{\bsF}{z}$ such that $\cK_1=g(z)\cK_2$. 
Clearly, since both operators are defined over $\Lfml{\bsF}{z}$, one has $g(z)\in\Lfml{\bsF}{z}$.\\

The irreducibility result does not immediately carry over to the case of the $\tau$-minimal grading shift. 
Instead, we generalize \ref{thm:rational-k:1D} by proving that the above result on the one-dimensionality of the space of QSP intertwiners \eqref{eq:spectral-k-intertwiner-V} for the principal grading implies the one-dimensionality of the space of QSP intertwiners for the $\tau$-minimal grading shift.\\

Let ${\sf pr}: \aQ \to \Z$ be the group homomorphism defined by $\al_i \mapsto 1$ for all $i \in \wh I$ so that, in the notation of Section~\ref{ss:tau-grading}, $\Sigma_z^{\sf pr}$ denotes the principal grading shift.
Consider any extension of ${\sf pr}|_{\sfQ}$ and $s_\tau|_{\sfQ}$ to group homomorphisms $\sfP\to\Q$, also denoted ${\sf pr}$ and $s_\tau$, respectively.
Note that the extended ${\sf pr}$ and $s_\tau$ will in fact take images in $\frac{1}{m} \Z$ for some positive integer $m$.
We shall regard $V$ as a $\sfP$-graded vector space.
Let $h$ be the Coxeter number of $\fkg$ and define 
\begin{equation*}
h_\tau = \frac{h}{|\{ 0,\tau(0) \}|} = \begin{cases} h & \text{if } \tau(0)=0 \\ \frac{h}{2} & \text{if } \tau(0) \ne 0 \end{cases} \qq \in \tfrac{1}{2}\Z
\end{equation*}
Let $M_{V}(z)$ be the linear operator on $V(z^{1/m}) \subset V\{z\}$ given by 
\begin{equation*}
M_{V}(z)v_{\la}=z^{\sf pr(\la) -h_\tau s_\tau(\la)}v_{\la}
\end{equation*}
for any weight vector $v_\la$ ($\la \in\sfP$).
Since $h=\operatorname{ht}(\drv{})$, it follows that $M_V(z)$ intertwines the {\em principal} shifted action with the {\em $\tau$-minimal }shifted action, \ie 
\begin{equation*}
\Ad(M_V(z)) \circ \pi^{s_{\tau}}_{z^{h_\tau},V} = \pi^{\sf pr}_{z,V}.
\end{equation*}
Finally, we observe that the linear map sending $K_{V}^{\sf pr}(z)$ to
\begin{equation*}
	K_{V}^{s_{\tau}}(z)\coloneqq M_{\twistrep{V}{\psi}}(z^{-1})^{-1} \cdot K_{V}^{\sf pr}(z) \cdot M_V(z)  
\end{equation*}
is an isomorphism between the spaces of principal and $\tau$-minimal QSP intertwiners. Thus, the result follows from Corollary \ref{cor:QSP:irreducible}. 


\subsection{Proof of Theorem~\ref{thm:rational-k}, parts \ref{thm:rational-k:factorization} and \ref{thm:rational-k:RE}}\label{ss:proof-rational-2}

Part \ref{thm:rational-k:factorization} follows immediately from part \ref{thm:rational-k:1D} and Lemma \ref{lem:rational-intertwiner}. 
Namely, let $\trKM{}{z}\in\End(V)(z)$ be any rational solution of \eqref{eq:spectral-k-intertwiner-V}, whose existence is guaranteed by Lemma~\ref{lem:rational-intertwiner}. 
By \ref{thm:rational-k:1D}, there exists $\wt{g}_V(z)\in\Lfml{\bsF}{z}$ such that the identity $\sKM{V}{z}= \wt{g}_{V}(z)\cdot\trKM{V}{z}$ holds. 
Then, by factoring out every pole and common zero of $\trKM{V}{z}$, we obtain the desired factorization $\sKM{V}{z}= {g}_{V}(z)\cdot\rKM{V}{z}$, where $\rKM{V}{z}$ is a non-vanishing operator-valued polynomial.\\

It remains to prove part \ref{thm:rational-k:RE}.
Let $\rRM{VW}{z}=f_{VW}(z)^{-1}\sRM{VW}{z}$ be the trigonometric R-matrix discussed in Thm.~\ref{thm:rational-R}. 
Then, \ref{thm:rational-k:RE} reduces to prove that
\begin{equation}\label{eq:scalar-identity}
f_{\twistrep{W}{\psi}\, \twistrep{V}{\psi}}(\tfrac{w}{z})\, g_{V}(z)\, f_{\twistrep{V}{\psi}\, W}(zw)\, g_{W}(w) = g_{V}(z)\, f_{\twistrep{W}{\psi}\, V}(zw)\, g_{W}(w) \, f_{VW}(\tfrac{w}{z})
\end{equation}
Note that, by \cite[Thm. 8.10]{AV22}, we have $\RM{21}^{\psi\psi} = J_{21} \cdot \RM{} \cdot J^{-1}$. 
Therefore,
\begin{align*}
	\sRM{\twistrep{W}{\psi}\, \twistrep{V}{\psi}}{\tfrac{w}{z}}_{21} = J_{21} \cdot \sRM{VW}{\tfrac{w}{z}} \cdot J^{-1} = f_{VW}(\tfrac{w}{z})\rRM{\twistrep{W}{\psi}\, \twistrep{V}{\psi}}{\tfrac{w}{z}}_{21}\, ,
\end{align*}
\ie $f_{\twistrep{W}{\psi}\, \twistrep{V}{\psi}}(\tfrac{w}{z})=f_{VW}(\tfrac{w}{z})$. 
Similarly, replacing ${V}$ with $\twistrep{V}{\psi}$, we get 
$f_{\twistrep{V}{\psi}\, W}(zw)=f_{\twistrep{W}{\psi}\, \twistrep{V}{\psi^2}}(zw)$\,.
Since $\gau\in\GQSP$, it follows that $\psi^2$ is the identity on $\Uqahp$ and hence the pullback by $\psi^2$ has no impact
on the normalization, \ie  
\begin{align*}
	f_{\twistrep{V}{\psi}\, W}(zw)=f_{\twistrep{W}{\psi}\, \twistrep{V}{\psi^2}}(zw) = f_{\twistrep{W}{\psi}\, V}(zw)\, .
\end{align*}
The result follows. 
\qedhere

\begin{remark}\label{rmk:rational:2}
	The definition of the set $\GQSP$ given in Section~\ref{ss:spectral-k} may appear at first quite \emph{ad hoc}. 
	However, for the proofs of Theorems~\ref{thm:spectral-k} and \ref{thm:rational-k} to work, a gauge
	transformation $\gau\in\Gg$ is required to satisfy the following three conditions:
	\vspace{0.25cm}
	\begin{enumerate}[label=(G\arabic*), start=1]\itemsep0.25cm
		\item\label{cond:G1} $\gau\cdot\bt{\tsat}^{-1}$ is shift-invariant;
		\item\label{cond:G2} $\gau\cdot\parc{}^{-1}$ descends to an operator on shifted representations;
		\item \label{cond:G3} $\psi^2$ is the identity on $\Uqahp$, where $\psi=\Ad(\gau)\circ\tsat_{q}^{-1}$.
	\end{enumerate}
	\vspace{0.25cm}
	By observing the action of $\gau$ on the Cartan subalgebra at $q=1$, one checks that the 
	conditions \ref{cond:G1}, \ref{cond:G2}, \ref{cond:G3} essentially determine $\GQSP$.
	\rmkend
\end{remark}

\subsection{Unitary K-matrices for $\psi$-involutive modules}\label{ss:unitarity}
Let $V,W\in\ModfdUqLg$ be irreducible representations. 
By Theorem~\ref{thm:rational-R}, the trigonometric R-matrix $\rRM{VW}{z}$ satisfies the unitarity condition $\rRM{VW}{z}^{-1}=(1\, 2)\circ\rRM{WV}{z^{-1}}\circ(1\, 2)$. 
Due to the lack of a canonical eigenvector, the analogous result for K-matrices requires a different approach. 
In this section, we discuss the case of $\UqLg$-modules such that $\twistrep{V}{\psi^2}=V$, while in Section~\ref{ss:unitarity-bis} we discuss
the special case $\twistrep{V}{\psi}=V$.

\begin{prop}\label{prop:unitarity}
	Let $V\in\ModfdUqLg$ be an irreducible module such that $\twistrep{V}{\psi^2}=V$.
	There exist non-vanishing trigonometric K-matrices $\rKM{V}{z},\rKM{\twistrep{V}{\psi}}{z} \in \End(V)(z)$ such that
	\vspace{0.25cm}
	\begin{enumerate}[label=$(U\arabic*)$, start=1]\itemsep0.25cm
		\item \label{eq:unitarity-K-1} $\rKM{V}{z}^{-1}=\rKM{\twistrep{V}{\psi}}{z^{-1}}$\,;
		\item \label{eq:unitarity-K-2}
		if $V(\zeta)$ is QSP irreducible for some $\zeta\in\bsF^{\times}$, then $\rKM{V}{\zeta}$ is well-defined and invertible. 
	\end{enumerate}
\end{prop}


\begin{proof}
	Let $\rKM{V}{z}$ and $\rKM{\twistrep{V}{\psi}}{z}$
	be any two polynomial trigonometric K-matrices for $V$ and $\twistrep{V}{\psi}$, respectively. The composition 
	\begin{equation*}
		\begin{tikzcd}[column sep=1.75cm]
			V(z) \arrow[r, "\rKM{V}{z}"] & 
			\twistrep{V}{\psi}(z^{-1}) \arrow[r, "\rKM{\twistrep{V}{\psi}}{z^{-1}}"] & 
			\twistrep{V}{\psi^2}(z)=V(z) 
			{}	\end{tikzcd}
	\end{equation*}	
	is a QSP intertwiner. Therefore, by Theorem~\ref{thm:rational-k}, we must have
	\begin{equation} \label{weakunitarity}
		\rKM{\twistrep{V}{\psi}}{z^{-1}}\rKM{V}{z}=f(z)\id_{V(z)}\, ,
	\end{equation}
	for some nonzero Laurent polynomial $f(z)\in\bsF[z,z^{-1}]$. 
	Thus \ref{eq:unitarity-K-1} follows by rescaling either $\rKM{V}{z}$ or $\rKM{\twistrep{V}{\psi}}{z^{-1}}$. 
	Note that the rescaled operators are non-vanishing $\End(V)$-valued rational functions satisfying the unitarity condition $\rKM{V}{z}^{-1}=\rKM{\twistrep{V}{\psi}}{z^{-1}}$. 
	 
	To prove \ref{eq:unitarity-K-2}, let $m\geqslant 0$ be the order of the pole of $\rKM{V}{z}$ at $z=\zeta$. The map
	\[
	\lim_{z\to \zeta}(z-\zeta)^m\rKM{V}{z}
	\]
	is a nonzero intertwiner $V(\zeta)\to \twistrep{V}{\psi}(\zeta^{-1})$ and therefore invertible.
	By unitarity, this implies that $\rKM{\twistrep{V}{\psi}}{z^{-1}}$ must have a zero of order $m$ at $z=\ze$. Since the latter is non-vanishing, we conclude that $m=0$ and \ref{eq:unitarity-K-2} follows.
\end{proof}

\begin{remark} \label{rmk:unitary} 
	Recall the twist $\psi=\chev\circ\tau$ in the semi-standard case, see Example \ref{rmk:distinguished} \ref{rmk:distinguished:2}.
	Since $\psi$ is involutive, the condition $\twistrep{V}{\psi^2}=V$ is satisfied for {\em any} $V\in\ModfdUqLg$. 
	In general, however, some further adjustment is required.
	For instance, we show in Section~\ref{s:recover-RE} that the distinguished twisting operator $\psi_0$  defined by \eqref{eq:res-rank-twist-op} acts (up to shift) as an involution on a certain class of representations, generalizing the case of $U_q(\wh{\mathfrak{sl}}_2)$ discussed in \cite[Sec.~9]{AV22}.
		\hfill \rmkend
	\end{remark}
	
	\subsection{Unitary K-matrices for $\psi$-fixed modules}\label{ss:unitarity-bis}
	In the special case $\twistrep{V}{\psi}=V$ we have the following refinement of Proposition \ref{prop:unitarity}, in the spirit of Remark~\ref{rmk:spectral-R-matrix} \ref{rmk:Rregularity}.
	The proof follows closely \cite[Sec.~6.3]{RV16}.
	
\begin{prop} \label{prop:unitarity:case2}
	Let $V\in\Modfd{\UqLg}$ be an irreducible representation such that $\twistrep{V}{\psi}=V$.
	Then there exists a non-vanishing trigonometric K-matrix $\rKM{V}{z} \in \End(V)(z)$ such that
	\vspace{0.25cm}
	\begin{enumerate}[label=$(U\arabic*')$]\itemsep0.25cm
		\item \label{eq:unitarity-K:case2} 
		$\rKM{V}{z}^{-1} = \rKM{V}{z^{-1}}$\,;
		\item \label{eq:unitarity-K:case2-bis} 
		if $V(\zeta)$ is QSP irreducible for some $\zeta\in\bsF^{\times}$, then $\rKM{V}{\zeta}$ is well-defined and invertible\,;
		\item \label{eq:unitarity-K:case2-ter} 
		if $V(\pm1)$ is QSP irreducible, then $\rKM{V}{\pm1}=\id_V$\,.
	\end{enumerate}
\end{prop}
	
	\begin{proof}
		Proceeding as in the proof of Proposition~\ref{prop:unitarity}, 
		we consider any polynomial trigonometric K-matrix $\rKM{V}{z}$ for $V$ satisfying the identity
		\begin{equation} \label{weakunitarity:case2}
			\rKM{V}{z^{-1}} \rKM{V}{z} = f(z) \id_{V(z)}\,,
		\end{equation}
		for some nonzero Laurent polynomial $f(z)\in\bsF[z,z^{-1}]$.
		Conjugation by $\rKM{V}{z^{-1}}$ then yields $f(z)=f(z^{-1})$. 
		As a consequence, the roots of $f$ are nonzero and, counting multiplicities, the set of roots of $f$ is invariant under inversion.
		Therefore, $f(z)=g(z)g(z^{-1})$ for some polynomial $g(z) \in \bsF[z]$. 
		By replacing $\rKM{V}{z}$ with $g(z)^{-1}\rKM{V}{z}$, we get \ref{eq:unitarity-K:case2}. 
		Part \ref{eq:unitarity-K:case2-bis} follows as in the proof of Proposition~\ref{prop:unitarity}.\\ 
		
		It remains to prove \ref{eq:unitarity-K:case2-ter}.
		Assume first that $V(\sigma)$ is QSP irreducible, while $V(-\sigma)$ is not, for some $\sigma\in\{\pm1\}$. 
		Then, $\rKM{V}{\sigma}^2=\id_V$ and therefore, by QSP irreducibility, $\rKM{V}{\sigma} = \veps \, \id_V$ with $\veps^2=1$. 
		By replacing $\rKM{V}{z}$ with $\veps \, \rKM{V}{z}$, we get \ref{eq:unitarity-K:case2-ter}. 
		Assume now that $V(1)$ and $V(-1)$ are both QSP irreducible. 
		As in the previous case, we get $\rKM{V}{\pm 1} = \veps_\pm \, \id_V$ with $\veps_{\pm}^2=1$. 
		By replacing $\rKM{V}{z}$ with $\veps_+ z^{\frac{1-\veps_+\veps_-}{2}} \, \rKM{V}{z}$, we get \ref{eq:unitarity-K:case2-ter}.
	\end{proof}
	
	\begin{remark}
	Proposition~\ref{prop:unitarity:case2} readily generalizes to the case $V^\psi\simeq V$. \hfill\rmkend
	\end{remark}
	
	
	\subsection{Abelian quasi-K-matrices and diagonal K-matrices}\label{ss:abelian-k-mx}
	The normalizations introduced in Propositions~\ref{prop:unitarity} and \ref{prop:unitarity:case2} become canonical in presence of a distinguished
	eigenvector.
	This is the case for restricted rank one QSP subalgebras (\cf~Remark~\ref{rmk:KM:pseudoinvolution} \ref{rmk:KM:pseudoinvolution:3}).\\
	
	Indeed, for $\tsat$ of restricted rank one, the quasi-K-matrix defined in Theorem \ref{thm:av-k-mx} is given purely in terms of imaginary root vectors, since it has the form $\QK{}= 1+\sum_{m \in \bbZ_{>0}}	\QK{m\del}$ with $\QK{m\del} \in \Uqanp_{m\del}$. 
	In particular, it preserves the weight spaces of any finite-dimensional $\UqLg$-module.
	
	\begin{prop}\label{prop:res-rank-one}
		Assume that $\tsat$ has restricted rank one and $\shift{z}(\gau)=\gau$.
		Let $V$ be an irreducible finite-dimensional $\UqLg$-module with $\ell$-highest weight vector $v_{0}\in V$. 
		There is a unique  non-vanishing trigonometric K-matrix 
		$\rKM{V}{z}\in\End(V)(z)$ regular at $z=0$ such that 
		\begin{align*}
			\rKM{V}{z}v_{0}=\gau\cdot v_{0}\, .
		\end{align*}
	\end{prop}
	
	\begin{proof}
		Recall that	$\psi=\Ad(\gau)\circ\tsat_q^{-1}$ and $K_\psi = \gau \cdot \Parc^{-1}\cdot\QK{}$.
		Note that $v_0$ is an eigenvector of the actions of $\QK{}$ and $\Parc$ with nonzero eigenvalue.
		Note that the quasi-K-matrix $\QK{}$ descends to the only spectral component of $\sKM{V}{z}$.
		In particular, there exists ${g}_V(z)\in\fml{\bsF}{z}^{\times}$ such that $\QK{V}(z) v_0 = {g}_V(z) \Parc \cdot v_0$, where $\QK{V}(z)\coloneqq\pi_{V,z}(\QK{})$.
		It follows that $\rKM{\psi}{z}\coloneqq g_V(z)^{-1}\cdot\sKM{ V}{z}$ is the unique intertwiner
		$\Lfml{V}{z}\to\Lfml{\twistrep{V}{\psi}}{z^{-1}}$ such that $\rKM{V}{z} v_0=\gau\cdot v_0$\, .
		Therefore, $\rKM{V}{z}\in\End(V)(z)$ is a non-vanishing rational operator regular at $z=0$.
	\end{proof}
	
	\begin{remarks}\hfill
		\begin{enumerate}[leftmargin=3em]\itemsep0.25cm
			\item 
			If $\gau$ is a weight zero operator, then $\rKM{V}{z}$ preserves the weight spaces of $V$. If the weight spaces are all one-dimensional, this yields a \emph{diagonal} trigonometric K-matrix.
			\item
			A similar result can be achieved through normalization with respect to any vector which is a simultaneous eigenvector for the imaginary root vectors. In particular, one
			can consider any {extremal} vector (cf. \cite{Kas02, Cha02}). \hfill\rmkend
		\end{enumerate}
	\end{remarks}
	
By combination with Propositions~\ref{prop:unitarity} and \ref{prop:unitarity:case2}, we get the following
	
\begin{corollary}
	Assume that $\tsat$ has restricted rank one, $\gau$ is weight zero and $\shift{z}(\gau)=\gau$.
	\vspace{0.25cm}
	\begin{enumerate}\itemsep0.25cm
		\item Let $V\in\ModfdUqLg$ be such that $V^{\psi^2}=V$ and consider the trigonometric K-matrices normalized by the conditions
		\begin{gather}
			\rKM{V}{z}v_{0}=v_{0}=\rKM{\twistrep{V}{\psi}}{z}v_0\,.
		\end{gather}
		Then $\rKM{V}{z}^{-1}=\rKM{\twistrep{V}{\psi}}{z^{-1}}$ holds. 
		Moreover, if $V(\zeta)$ is QSP irreducible for some $\zeta\in\bsF^{\times}$, then $\rKM{V}{\zeta}$ is well-defined and invertible.
		\item Let $V\in\ModfdUqLg$ be such that $V^{\psi}=V$ and
		consider the diagonal trigonometric K-matrix normalized by the condition
		$\rKM{V}{z}v_{0}=v_{0}$. Then, if $V(\pm1)$ is QSP irreducible,
		$\rKM{V}{\pm1}=\id_V$.
	\end{enumerate}
\end{corollary}

\begin{proof}
	It suffices to note that, since $\gau$ is weight zero, there exists ${g}_V(z)\in\fml{\bsF}{z}^{\times}$ such that $\sKM{V}{z}v_0=g_V(z)v_0$,
	therefore $\rKM{V}{z}=g_V(z)^{-1}\sKM{V}{z}$. 
	Moreover, $v_0\in\twistrep{V}{\psi}$ is still a simultaneous eigenvector for the imaginary root vectors.
	Then the results follow as in Propositions~\ref{prop:unitarity} and \ref{prop:res-rank-one}.
\end{proof}
	


\section{Generic restricted irreducibility}\label{s:irreducibility}

In this stand-alone section, we discuss various problems of \emph{restricted irreducibility} for finite-dimensional $\UqLg$-modules. 
More precisely, we consider certain proper subalgebras $\nA \subset \UqLg$ and 
we describe the irreducible representations which are generically irreducible under restriction, 
\ie those $V\in\ModfdUqLg$ such that $V(z)$ is irreducible as a module over $\nA(z)\coloneqq\nA\ten\bsF(z)$.
Recall that the definition of $V(z)$ depends on the choice of a grading shift.
For a class of subalgebras, which we refer to as \emph{modified nilpotent}, we prove that every irreducible $\UqLg$-module remains, 
with respect to the principal grading shift, generically irreducible under restriction (Theorem \ref{thm:deformedN:irreducible}). 
As a corollary, we prove that every irreducible $\UqLg$-module is generically QSP irreducible (Corollary~\ref{cor:QSP:irreducible}).\\

We shall use the shorthand notations $\wh U^+$, $\wh U^0$ and $\wh U^-$ to denote the subalgebras of $\UqLg$ generated by $E_i$, $K_i^{\pm 1}$, and $F_i$, respectively, for all $i \in \wh I$.
Set $\wh U^{\geqslant 0} \coloneqq  \wh U^+ \cdot \wh U^0$.
We shall consider only the principal grading shift $\Sigma^{\sf pr}_z: \UqLg \to \UqLg[z,z^{-1}]$, see Section~\ref{ss:tau-grading}. 
As before, for any finite-dimensional $\UqLg$-module $V$, we denote by $V(z) = V \ten \bsF(z)$ the corresponding shifted representation.

\subsection{Generic irreducibility for modified nilpotent subalgebras}

A subalgebra of $\UqLg$ is a \emph{modified (lower) nilpotent subalgebra} if it is generated by elements $\wt F_i$ satisfying $\wt F_i - F_i \in \wh U^{\geqslant 0}$ for all $i \in \wh I $.
The main result of this section is the following statement, which we prove in Section~\ref{sec:gnrc:irred:modified}.

\begin{theorem} \label{thm:deformedN:irreducible}
Every irreducible finite-dimensional $\UqLg$-module is generically irreducible under restriction to a modified nilpotent subalgebra.
\end{theorem}

As in Section~\ref{s:affine-qsp}, let $\theta$ be a pseudo-involution of $\wh \fkg$ of the second kind with associated generalized Satake diagram $(X,\tau)$. Let $(\bm \ga,\bm \si) \in {\bm \Gamma}\times {\bm \Sigma}$ and set $\wt F_i=\Bg{i}$.
From $\theta_q(F_i) \in \wh U^+ \cdot K_i^{-1}$ for all $i \not\in X$ and $\Bg{i} = F_i$ if $i \in X$, we obtain $\wt F_i - F_i \in \wh U^{\geqslant 0}$.
Thus, every QSP subalgebra contains a modified nilpotent subalgebra.
Hence, Theorem \ref{thm:deformedN:irreducible} yields the following result of generic QSP irreducibility.

\begin{corollary} \label{cor:QSP:irreducible}
	Every irreducible finite-dimensional $\UqLg$-module is generically irreducible under
	restriction to a QSP subalgebra.
\end{corollary}

\subsection{Restricted irreducibility for nilpotent subalgebras}
The proof of Theorem \ref{thm:deformedN:irreducible} crucially relies on
the following result: every finite-dimensional irreducible $\UqLg$-module remains irreducible under restriction to $\wh U^+$. 
To our knowledge, this result is not explicitly stated in the literature. However, the analogous result for $\wh U^{\geqslant 0}$ is known as stated below.

\begin{prop}[{\cite[Prop.~3.5]{HJ12} or \cite[Thm. 2.3 (ii) for $\eps = (1)^n$]{Bow07}}] \label{prop:N:irreducible:prep}
	Every irreducible finite-dimensional $\UqLg$-module remains irreducible under restriction to $\wh U^{\geqslant 0}$.
\end{prop}

In fact, the proof of \cite[Prop.~3.5]{HJ12}, which relies on the Drinfeld presentation of $\UqLg$, can be easily strengthened to show that every finite-dimensional irreducible $\UqLg$-module remains irreducible under restriction to $\wh U^+$.
Nonetheless, in order to be consistent with the rest of the paper, we provide a proof in terms of the Drinfeld-Jimbo presentation.\\

Let $V$ be a fixed finite-dimensional irreducible $\UqLg$-module. By restriction to $\Uqg$, $V$ decomposes into irreducible $\Uqg$-modules, yielding a weight decomposition $V = \bigoplus_{\la \in \Plat} V_\la$, where 
\begin{equation*}
	V_\la \coloneqq \{ v \in V \, \vert\, \forall i \in I\,,\,K_i \cdot v = q_i^{\la(h_i)} v \}.
\end{equation*}
Note that $K_0$ acts on each $V_\la$ as multiplication by $q_i^{\la(h_0)} = q^{-(\la,\hrt)}$.
Let $\ol{\cdot}: \aQ \to \Qlat$ be the unique $\Z$-linear projection such that $\ol{\al_i}=\al_i$ ($i \in I$) and $\ol{\al_0} = -\hrt$ (so $\ol{\del}=0$).
For $i \in \wh I$ and $\la \in P$, we have 
\begin{equation*}
	E_i \cdot V_\la \subseteq V_{\la + \ol{\al_i}} 
	\qq\mbox{and}\qq F_i \cdot V_\la \subseteq V_{\la - \ol{\al_i}}.
\end{equation*}
Set 
$\mathsf{Supp}(V) \coloneqq  \{ \lambda \in \Plat \, | \, V_\lambda \ne \{0\} \}$.
From Proposition \ref{prop:N:irreducible:prep} we deduce some useful technical results, the first of which directly generalizes \cite[Cor.~2.7]{CG05}.

\begin{lemma} \label{lem:N:irreducible:prep}
	Let $\la \in \Supp(V)$ and $0\neq v_\la \in V_\la$.
	\begin{enumerate}
		\item \label{lem:N:irreducible:prep:1} $V = \wh U^+ \cdot v_\la$.
		\item \label{lem:N:irreducible:prep:2} For any $\mu \in \Supp(V)$, there exist $\ell \in \Z_{\geqslant 0}$ and $\bm i \in \wh I^\ell$ such that $E_{i_1} \cdots E_{i_\ell} \cdot v_\la$ is nonzero and has weight $\mu$.
	\end{enumerate}
\end{lemma}

\begin{proof}
	Part \ref{lem:N:irreducible:prep:1} follows from Proposition \ref{prop:N:irreducible:prep}, the decomposition $\wh U^{\geqslant 0} = \wh U^+ \cdot \wh U^0$ and the fact that $\wh U^0 \cdot v_\la = \bbF v_\la$.
	To prove part \ref{lem:N:irreducible:prep:2}, note that \ref{lem:N:irreducible:prep:1} implies the existence of $x \in \wh U^+$ such that $0\neq x \cdot v_\la \in V_\mu$.
	Since each monomial in the $E_i$ ($i \in \wh I$) sends $v_\la$ into a weight space, we may assume without loss of generality that all monomials occurring in $x$ map $v$ to $V_\mu$. Since at least one of these monomials does not annihilate $v_\la$, we obtain the result.
\end{proof}

For any $v \in V$, let $\mathsf{Supp}(v) \subseteq \mathsf{Supp}(V)$ be 
defined by the property 
\[
v = \sum_{\la \in \mathsf{Supp}(v)} v_\la
\]
with nonzero $v_\lambda\in V_\lambda$.
Note that $|\mathsf{Supp}(v)|=1$ if and only if $v$ is a weight vector.
For any $i \in \wh I$ we have
\begin{equation}\label{support:shift}
	\mathsf{Supp}(E_i \cdot v) \subseteq {\sf Supp}(v) + \ol{\al_i} 
	\qq\mbox{and}\qq 
	\mathsf{Supp}(F_i \cdot v) \subseteq \mathsf{Supp}(v) - \ol{\al_i}.
\end{equation}
Hence,
\begin{equation} \label{width:bound}
	|\mathsf{Supp}(E_i \cdot v)|\,,\,|\mathsf{Supp}(F_i \cdot v)| \leqslant |\mathsf{Supp}(v)|.
\end{equation}

By \cite[Sec.~3]{CP94a}, also see \cite[Sec.~2.3]{CG05}, for a given irreducible finite-dimensional $U_q(L\fkg)$ module-$V$ there exists $\lambda_0 \in P$ such that $\Supp(V)$ is contained in $\lambda_0 -\Qp$, $E_i \cdot V_{\lambda_0} = \{0\}$ for all $i \in I$ and $F_0 \cdot V_{\lambda_0} = \{0\}$; moreover $\dim(V_{\lambda_0})=1$.
We define the \emph{depth} of a weight vector as follows.
For any $\la \in \Supp(V)$ and $0\neq v_\la \in V_\la$, set
\begin{equation} \label{depth:def1}
	d^+(v_\la) \coloneqq \min\{ \ell \in \Z_{\geqslant 0} \, | \, \exists \bm i \in \wh I^\ell \;
	\mbox{such that}\; 0\neq E_{i_1} \cdots E_{i_\ell} \cdot v_\la \in V_{\la_0} \}.
\end{equation}
Note that, by Lemma \ref{lem:N:irreducible:prep} \ref{lem:N:irreducible:prep:2}, $d^+(v_\la)\in\bbZ_{\geqslant0}$ is 
well-defined.
More generally, for any $0\neq v \in V$ with $v=\sum_{\lambda\in{\sf Supp}(v)}v_{\lambda}$, we set
$d^+(v) \coloneqq \min \{ d^+(v_\la) \, | \, \la \in \mathsf{Supp}(v) \}$.
Note that $d^+(v) = 0$ if and only if $\la_0 \in {\sf Supp}(v)$.

\begin{lemma} \label{lem:N:irreducible:prep2}
	Let $0\neq v \in V$.
	If $d^+(v)>0$, there exists $i \in \wh I$ such that $E_i \cdot v \ne 0$ and $d^+(E_i \cdot v)<d^+(v)$.
\end{lemma}

\begin{proof}
	Set $\ell \coloneqq  d^+(v) \in \Z_{>0}$.
	By definition, $\la_0 \notin {\sf Supp}(v)$ and there exists $\la \in {\sf Supp}(v)$ such that $d^+(v_\la) = \ell$.
	By Lemma \ref{lem:N:irreducible:prep} \ref{lem:N:irreducible:prep:2}, there exists $\bm i \in \wh I^\ell$ such that $0\neq E_{i_1} \cdots E_{i_\ell} \cdot v_\la \in V_{\la_0}$.
	Set $i=i_\ell$.
	Then $d^+(E_i \cdot v_\la) = \ell-1$.
	The relation \eqref{support:shift} and the definition of $d^+$ yield $d^+(E_i \cdot v) < \ell$.
\end{proof}

Finally, we obtain the following result of restricted irreducibility.

\begin{theorem} \label{thm:N:irreducible}
	Every irreducible finite-dimensional  representation of $\UqLg$ remains irreducible under restriction to $\wh U^+$.
\end{theorem}

\begin{proof}
	We will show that $\wh U^+ \cdot v = V$ for all $v \in V \backslash \{0\}$.
	By Lemma \ref{lem:N:irreducible:prep} \ref{lem:N:irreducible:prep:1}, it suffices to show that $\wh U^+ \cdot v$ contains a weight vector.
	We proceed by induction on $s \coloneqq  |\Supp(v)| \in \Z_{>0}$.
	For $s=1$, it is clear.
	For $s>1$, suppose that there exists $v \in V$ such that $|\mathsf{Supp}(v)|=s$ and $\wh U^+ \cdot v$ does not contain a weight vector. Choose one such $v$ with minimal $d^+(v)$.
	\vspace{0.25cm}
	\begin{itemize}[leftmargin=1em]\itemsep0.25cm
		\item
		Suppose $d^+(v)=0$.
		Then $\la_0 \in {\sf Supp}(v)$.
		Since $|{\sf Supp}(v)|>1$, we may choose $\mu \in {\sf Supp}(v)$ with $\mu\neq \la_0$.
		By Lemma \ref{lem:N:irreducible:prep} \ref{lem:N:irreducible:prep:2}, there exists $\bm i \in \wh I^\ell$ with $\ell = d^+(\mu)$ such that $0\neq E_{i_1} \cdots E_{i_\ell} \cdot v_\mu \in V_{\la_0}$.
		Thus, we must have $\mu + \ol{\al_{i_1} + \ldots + \al_{i_\ell}} = \la_0$ with $\ol{\al_{i_1} + \ldots + \al_{i_\ell}} \ne 0$.
		Since ${\sf Supp}(V) \subset \la_0 - \Qp$, we obtain $\ol{\al_{i_1} + \ldots + \al_{i_\ell}} \in \Qp$.
		Since $\la_0 + \ol{\al_{i_1} + \ldots + \al_{i_\ell}} \in \la_0 +\Qp$ and ${\sf Supp}(V) \subset \la_0 -\Qp$, we obtain $\la_0 + \ol{\al_{i_1} + \ldots + \al_{i_\ell}} \notin {\sf Supp}(V)$.
		Hence $E_{i_1} \cdots E_{i_\ell} \cdot v_{\la_0} =0$.
		Therefore, $|{\sf Supp}(E_{i_1} \cdots E_{i_\ell} \cdot v)| < |{\sf Supp}(v)|$.
		\item
		Suppose $d^+(v)>0$.
		By Lemma \ref{lem:N:irreducible:prep2}, there exists $i \in \wh I$ such that $E_i \cdot v \ne 0$ and $d^+(E_i \cdot v) < d^+(v)$.
		From the minimality of $v$ and \eqref{width:bound} we deduce that $|{\sf Supp}(E_i \cdot v)| < |{\sf Supp}(v)|$.
	\end{itemize}
	\vspace{0.25cm}
	In either case, we obtain an inequality of the form $|{\sf Supp}(E_{i_1} \cdots E_{i_\ell} \cdot v)|<|{\sf Supp}(v)|$. 
	Noting that $\wh U^+ \cdot (E_{i_1} \cdots E_{i_\ell} \cdot v) \subseteq \wh U^+ \cdot v$, we obtain a contradiction. 
	Thus, $\wh U^+ \cdot v$ necessarily contains a weight vector and the result follows.
\end{proof}

By applying the Chevalley involution, we get the negative counterpart of the previous result.

\begin{corollary} \label{crl:N:irreducible}
	Every irreducible finite-dimensional  representation of $\UqLg$ remains irreducible under restriction to $\wh U^-$.
\end{corollary}

\subsection{Proof of Theorem \ref{thm:deformedN:irreducible}} \label{sec:gnrc:irred:modified}

The proof of the main result of this section relies on the following basic fact, where $\bbF$ is an arbitrary field, $z$ an indeterminate and evaluation at 0 in the polynomial ring $\bbF[z]$ is denoted ${\sf ev}_0$.

\begin{lemma} \label{lem:linearalgebra}
	Let $V$ be a finite-dimensional $\bbF$-linear space and $\mathcal M(z)$ a nonzero $\bbF((z))$-linear subspace of $(\End V)(z)$ 
Denote the $\bbF$-linear subspace of $\mathcal M(z)$ of polynomial elements by $\mathcal M[z]$.
	Set $\mathcal M \coloneqq  {\sf ev}_0(\mathcal M[z]) \subseteq \End V$. 
If $V$ has no nontrivial proper $\mathcal M$-invariant subspace, then $V(z) = V\ten\bbF(z)$ 
has no nontrivial proper $\mathcal M(z)$-invariant subspace.
\end{lemma}

\begin{proof}
	Let $S(z)\neq\{0\}$ be an $\mathcal M(z)$-invariant subspace of $V(z)$ and consider $S[z] \coloneqq S(z) \cap V[z]$, which does not lie in the kernel of ${\sf ev}_0$.
	Since $S[z]$ is $\mathcal M[z]$-invariant, evaluation at 0 yields that ${\sf ev}_0(S[z])$ is a nontrivial $\mathcal M$-invariant subspace.
	We obtain that ${\sf ev}_0(S[z]) = V$. 
	By extension of scalars, we get $S(z) = V(z)$.
\end{proof}

We can now complete the proof of Theorem \ref{thm:deformedN:irreducible}.
Let $V$ be a finite-dimensional $U_q(L\fkg)$-module.
For a given modified nilpotent subalgebra of $U_q(L\fkg)$ with generators $\wt F_i$ ($i \in \wh I$), denote by $\mathcal M(z)$ the induced subalgebra of $(\End V)(z)$. 
Let $i \in \wh I$ be arbitrary. 
Since $\wt F_i - F_i \in \wh U^{\geqslant 0}$, the action of $z \wt F_i \in \UqLg(z)$ on $V(z)$ is polynomial in $z$.
Given the principal grading shift, using the notation of Lemma \ref{lem:linearalgebra}, the action of $z F_i$ on $V(z)$ is contained in $\mathcal M$ and descends to the action of $F_i$ on $V$.
By Corollary \ref{crl:N:irreducible}, $V$ contains no nontrivial proper $\mathcal M$-invariant subspace.
Now Lemma \ref{lem:linearalgebra} yields the desired result.


\section{Solutions of reflection equations}\label{s:recover-RE}

In this section, we apply the results from Section~\ref{s:rational} to produce new 
solutions of the two most common forms of the generalized reflection equation, the standard and transposed reflection equation. 
These solutions arise from distinguished trigonometric K-matrices on small modules and Kirillov-Reshetikhin modules.

\subsection{The standard and the transposed reflection equations}\label{ss:std-re}
Let $V$ and $W$ be two irreducible finite-dimensional $\UqLg$-modules and denote by $\rRM{VW}{z}$ is the trigonometric R-matrix from Theorem \ref{thm:rational-R}.
On the tensor product $\shrep{V}{z} \ten \shrep{W}{w}$ we consider the standard reflection equation 
\begin{equation} \label{eq:rational-original-RE}
	\begin{aligned}
		& \rRM{WV}{\tfrac{w}{z}}_{21} \cdot (\id_V \ten\rKM{W}{w} )\cdot \rRM{VW}{zw} \cdot (\rKM{V}{z} \ten \id_W) = \\
		& \qquad \qquad = (\rKM{V}{z} \ten \id_W) \cdot \rRM{WV}{zw}_{21} \cdot (\id_V \ten \rKM{W}{w}) \cdot \rRM{VW}{\tfrac{w}{z}}
	\end{aligned}
\end{equation}
where $\rKM{V}{z}$ and $\rKM{W}{w}$ are unknown operator-valued rational functions, 
and the transposed reflection equation
\begin{equation}
	\label{eq:rational-transposed-RE}
	\begin{aligned}
		& \rRM{WV}{\tfrac{w}{z}}_{21}^{{\sf t}_V {\sf t}_W} \cdot (\id_V \ten \trKM{W}{w}) 
			\cdot ( \rRM{VW}{zw}^{-1})^{\sf t_V} \cdot (\trKM{V}{z} \ten \id_W) = \\
		& \qquad \qquad = (\trKM{V}{z} \ten \id_W) \cdot (\rRM{WV}{zw}^{-1}_{21})^{\sf t_W} 
			\cdot (\id_V \ten \trKM{W}{w}) \cdot \rRM{VW}{\tfrac{w}{z}}\,
	\end{aligned}
\end{equation}
where $\trKM{V}{z}$ and $\trKM{W}{w}$ are unknown operator-valued rational functions and ${\sf t}_V$ and ${\sf t}_W$ denote the transposition on the corresponding component.

\subsection{Small modules}
A finite-dimensional $\UqLg$-module is {\em small} if it remains irreducible under restriction to $\Uqg\subset\UqLg$,
\ie it is an irreducible finite-dimensional $\Uqg$-module, whose action extends to one of $\UqLg$. 
In type $\sfA$, a $U_q(L\mathfrak{sl}_N)$-module is small if and only if it is an evaluation representation, see \cite{CP94b}. 
In other types, an exhaustive description of small modules is not known. 
However, the case of fundamental representations is studied in detail in \cite[Thm.~6.8]{CP94a}. 
Indeed, in this case we have the following sufficient condition.
Recall that $a_i$ indicates the multiplicity of the simple root $\al_i$ in the highest root $\vartheta$ of $\g$. 

\begin{lemma}
	Let $i \in I$. 
	If $a_i=1$ or $a_i=\iip{\vartheta}{\vartheta}/\iip{\rt{i}}{\rt{i}}$, the action of $\Uqg$ on the fundamental representation $V_{\om_i}$ extends to an action of $\UqLg$.
\end{lemma}

\begin{proof}
	By \cite[Prop.~12.1.17]{CP95}, under the same hypothesis the action of $\g$ on its fundamental representation extends to an action of the Yangian. 
	By \cite{GTL16}, the same result applies to fundamental representations of $\Uqg$.
\end{proof}

For classical types, this result applies to every fundamental representation in type $\sfA$ and $\sfC$, 
the vector and the spin representations in types $\sfB$ and $\sfD$.
For exceptional types, in terms of the labelling conventions of \cite{Bou68}, it applies to the fundamental representations $V_{\omega_1}$ and $V_{\omega_6}$ for $\sfE_6$, $V_{\omega_7}$ for $\sfE_7$,
$V_{\omega_4}$ for $\sfF_4$, and $V_{\omega_2}$ for $\sfG_2$.

\subsection{Amenable representations}\label{ss:amenable}
Let $V$ be a small finite--dimensional $\UqLg$-module.
By Jacobson's density theorem, the defining data of $V$ as a representation over $\UqLg$ can be expressed entirely in terms of $\Uqg$. 
Thus, there exist $u_0^{\pm}\in\Uqg$ such that $\pi_V(x_0^{\pm})=\pi_V(u_0^{\pm})$, where $x_0^+\coloneqq E_0$ and $x_0^-\coloneqq F_0$. 
In the following, we shall consider small representations where the elements $u_0^\pm$ have a prescribed expression.\\ 

Let $x_\vartheta^{\pm}\in\Uqn^\pm$ be a choice of root vectors corresponding to the highest root $\vartheta$ such that $\omega(x_\vartheta^{\pm})=-x_\vartheta^{\mp}$.
Note that $\oi_I(x_\vartheta^{\pm})=\veps \, x_\vartheta^{\pm}$
for some $\veps \in \{\pm1 \}$.

\begin{definition}
	A small finite-dimensional $\UqLg$-module $V$ is \emph{amenable} 
	if 
	\begin{equation}\label{eq:amenable}
		\pi_V(x_0^\pm)=\lambda_\mp\pi_V(x_\vartheta^\mp)
	\end{equation} 
	for some $\lambda_{\pm}\in\bsF^\times$.
\end{definition}

\begin{remark}
	In type $\sfA$, every evaluation representation is amenable.
	In the other classical types, one checks by
	direct inspection that the vector representation is amenable.
	More generally, we expect that every small finite-dimensional
	$\UqLg$-module is amenable.
	\rmkend
\end{remark}

\subsection{The twisting operator $\psi_0$}

Let $(X,\tau)\in\gsat{\aff{A}}$ be an affine generalized Satake diagram with pseudo-involution $\tsat$, $(\parc{},\pars{})\in\Parsetc\times\Parsets$, and $\Uqk\subset\Uqagp$ the corresponding QSP subalgebra. 
We consider the subdiagram $Y_0\subset \aIS$ and the automorphism $\eta_0:\aIS\to\aIS$ given by
\begin{equation}
	\eta_0(0)=\tau(0)\,
	\qq\mbox{and}
	\qq \eta_0\vert_{Y_0}=\oi_{Y_0}
\end{equation}
where $Y_0\coloneqq\aIS\setminus\{0,\tau(0)\}$.
By casework, $\eta_0$ is a diagram automorphism and $(Y_0,\eta_0)$ is a generalized Satake diagram. Note also that $\eta_0$ and $\tau$ commute. 
Finally, let $\zsat_0$ be the pseudo-involution
corresponding to $(Y_0,\eta_0)$ and choose a QSP admissible
twisting operator (\cf~Section~\ref{ss:spectral-k})
\begin{equation}\label{eq:res-rank-twist-op}
	\psi_0=\Ad(\gau_0)\circ\tsat_q^{-1}
\end{equation}
where $\gau_0\in\GQSP$ has the form $\gau_0\coloneqq\bt{\zsat_0}^{-1}\bt{\tsat}{\bm\beta}^{-1}$ for some ${\bm\beta}:\Plat\to\bsF^\times$.

\subsection{$\tau$-restrictable QSP subalgebras}\label{ss:tau-res}
An affine QSP subalgebra $\Uqk\subset\UqLg$ is $\tau$-{\em restrictable} if $\tau(0)=0$. In this case, the $\tau$-minimal grading shift coincides with the homogeneous grading shift, and $Y_0$ and $\eta_0$ are independent of $\tau$, since $Y_0=I$ and $\eta_0$ is the unique affine extension of the opposition involution $\oi_I$ on the finite Dynkin diagram. Moreover, in 
\eqref{eq:res-rank-twist-op} we choose $\bm\beta$ such that ${\bm\beta}(\rt{i})=1$ if $i\neq0$ and ${\bm\beta}(\rt{0})={\bm\gamma}^{-1}(\drv{})$.

\begin{lemma}\label{lem:twist-restrictable-case}
	Let $\Uqk\subset\UqLg$ be a $\tau$-restrictable affine QSP subalgebra and
	$V$ a small amenable finite-dimensional $\UqLg$-module. 
	There exists unique $c_V\in\bsF^\times$ such that $\twistrep{V}{\psi_0}=\twistrep{V}{\eta_0\tau}(c_V)$.
\end{lemma}

\begin{proof}
	First, we consider the case ${\bm\gamma}(\drv{})=1$.
	Since $\Uqk$ is $\tau$-restrictable, we have $Y_0=I$ and $\eta_0|_I = \oi_I$. Therefore, $\zsat_{0,q}$ is the identity on $\Uqg$ and $\twistrep{V}{\psi_0}=\twistrep{V}{\eta_0\tau}$ as $\Uqg$-modules.
	Since $V$ is small, there exist $u_0^{\pm}\in\Uqg$ such that $\pi_V(x_0^{\pm})=\pi_V(u_0^{\pm})$, where as before  $x_0^+\coloneqq E_0$ and $x_0^-\coloneqq F_0$.
	By construction, we have $\psi_0 =$ $\Ad(\bt{\zeta_0}^{-1}) \circ \om \circ \tau$ with $\bt{\zeta_0} = \tcorr{\zeta_0}\cdot\qWS{Y_0}$.
	Since $\qWS{Y_0}$ is supported on $U_q(\fkg)$ and $V$ is a $\Plat$-weight module, it follows that 
	\begin{equation*}
		\pi_V(\psi_0(x_0^{\pm})) = -\pi_V(\bt{\zeta_0}^{-1}) \cdot \pi_V(x_0^\mp) \cdot \pi_V(\bt{\zeta}) = -\pi_V(\Ad(\bt{\zeta_0}^{-1})(u_0^\mp)).
	\end{equation*}
	Since $\Ad(\bt{\zeta_0}^{-1})$ acts on $U_q(\fkg)$ as $\om \circ \oi_I$, we obtain $\pi_V(\psi_0(x_0^{\pm})) = -\pi_V((\omega \circ \oi_I)(u_0^{\mp}))$.
	Since $V$ is amenable, $u_0^\mp=\lambda_{\pm}x_{\vartheta}^{\pm}$ for some $\lambda_{\pm} \in \bsF^\times$, hence $\oi_I(u_0^{\pm})=\veps \, u_0^{\pm}$
	for some $\veps \in \{\pm1 \}$, see Section~\ref{ss:amenable}.
	Therefore, by setting $c\coloneqq\veps\lambda_-/\lambda_+$, we get $-(\omega \circ \oi_I)(u_0^{\mp})=c^{\pm1}u_0^{\pm}$.
	Finally this yields $\pi_V(\psi_0(x_0^{\pm}))=c^{\pm1}\pi_V(x_0^{\pm})$ and therefore $V^{\psi_0}=V^{\eta_0\tau}(c)$.
	Indeed, since the composition $\eta_0 \tau$ fixes the node 0, the identification holds for $E_0$ and $F_0$.
	The proof for $K_0$ is analogous and relies on the fact that $\pi_V(K_0) = \pi_V(K_\vartheta^{-1})$. The case ${\bm\gamma}(\drv{})\neq 1$
	is similar and its proof is omitted for simplicity.
\end{proof}

\begin{remarks}\label{rmk:shift}
	\hfill
	\begin{enumerate}[leftmargin=2em]\itemsep0.25cm
	\item \label{rmk:shift:1}
	The factor $c_V\in\bsF^\times$ determined by the identity $V^{\psi_0}=V^{\eta_0\tau}(c_V)$ 
	can be removed through a suitable shift of the representation $V$. 
	Namely, proceeding as in \cite[Sec.~9]{AV22}, one shows that there exists a unique $c\in \bsF^\times$ (up to a sign) 
	such that the shifted representation $V_c\coloneqq V(c)$ satisfies $V_c^{\psi_0}=V_c^{\eta_0\tau}$. 
	Therefore, up to a uniquely determined shift, for any small amenable finite-dimensional $\UqLg$-module $V$, one has $V^{\psi_0}=V^{\eta_0\tau}$.
	\item \label{rmk:shift:2}
	From the classification of generalized Satake diagrams of affine type in \cite[App.~A, Tables 5, 6 and 7]{RV21} it follows that
	a QSP subalgebra is $\tau$-restrictable if and only if $\tau$ is either the identity or $\eta_0$,
	except in type $\sfD^{(1)}_{2n}$, where there exist an involutive diagram automorphisms which fixes the affine node and differs from $\eta_0=\id$.
	\item \label{rmk:shift:3}
	An affine QSP subalgebra $\Uqk\subset\UqLg$ is {\em restrictable} if it is $\tau$-restrictable and $0\not\in X$.  
 	By Corollary~\ref{cor:rational-regular-k} the universal K-matrix defined in terms of the gauge element $\gau = \gau_0$ 
	descends to any finite-dimensional $\UqLg$-module $V$ to a formal operator $\sKM{V}{z}\in\fml{\End(V)}{z}$. 
	Moreover, in analogy with the case of the R-matrix, it follows by the uniqueness of the quasi-K-matrix 
	that $\sKM{V}{0}$ corresponds to the action on $V$ of the universal K-matrix associated to $U_q(\fkg)$ considered in \cite{BK19}. 
	\rmkend	
	\end{enumerate}
\end{remarks}

\subsection{Solutions of the diagrammatic reflection equation}\label{ss:our-solutions}
In the case of $\tau$-restrictable QSP subalgebras, the trigonometric K-matrices constructed in Theorem~\ref{thm:rational-k} on small amenable modules are solutions of a \emph{diagrammatic} reflection equation.

\begin{theorem}\label{thm:sol-diag-RE}
	Let $\Uqk\subset\UqLg$ be a $\tau$-restrictable affine QSP subalgebra.
	\vspace{0.25cm}
	\begin{enumerate}[leftmargin=2em]\itemsep0.25cm
	\item 
	Let $V$ be a small amenable finite-dimensional $\UqLg$-module. 
	Up to a uniquely determined shift in $V$ (cf. Remark~\ref{rmk:shift} \ref{rmk:shift:1}), there exists (up to a scalar multiple) a unique QSP intertwiner 
	\begin{equation} \label{QSPintertwiner:shifted}
		\rKM{V}{z}:\shrep{V}{z}\to\shrep{V^{\eta_0\tau}}{z^{-1}}
	\end{equation} 
	\item 
	Let $V, W$ be two small amenable finite-dimensional $\UqLg$-modules. 
	The operators $\rKM{V}{z}$, $\rKM{W}{w}$ satisfy the {diagrammatic} reflection equation
	\begin{align}\label{eq:bala-kolb-re}
		& \rRM{WV}{\tfrac{w}{z}}_{21} \cdot (\id\ten\rKM{W}{w}) \cdot {\bf R}_{V\,W}^{\eta_0\tau}(zw) \cdot (\rKM{V}{z}\ten \id) = \\
		& \qquad \qquad = (\rKM{V}{z} \ten \id) \cdot {\bf R}_{W\,V}^{\eta_0\tau}(zw)_{21} \cdot (\id\ten\rKM{W}{w}) \cdot \rRM{VW}{\tfrac{w}{z}}\, .
	\end{align}
	\end{enumerate}
\end{theorem}

\begin{proof}
	It is enough to observe that, by Lemma~\ref{lem:twist-restrictable-case} and Remark~\ref{rmk:shift} \ref{rmk:shift:1},
	$\twistrep{V}{\psi_0}=V^{\eta_0 \tau}$ (up to a suitable shift in $V$).
	Consider the trigonometric K-matrix $\rKM{V}{z}$ from Theorem~\ref{thm:rational-k} for the gauge element $\gau = \gau_0$. 
	It provides the desired intertwiner and the generalized reflection equation satisfied by $\rKM{V}{z}$ and $\rKM{W}{w}$ on $V(z)\ten W(w)$ reduces to \eqref{eq:bala-kolb-re}.
\end{proof}

\begin{remark}
	The same result applies for any irreducible finite-dimensional representation $V$ (not necessarily small) equipped with a distinguished isomorphism $V^{\psi_0}\simeq V^{\eta_0\tau}$.\hfill\rmkend
\end{remark}

In Sections~\ref{ss:small-std-RE}-\ref{ss:KR-RE}, we shall discuss several refinements of Theorem~\ref{thm:sol-diag-RE}.

\subsection{Solutions of the standard and the transposed reflection equations}\label{ss:small-std-RE}

Solutions of \eqref{eq:rational-original-RE} and \eqref{eq:rational-transposed-RE} on amenable modules are obtained as an immediate consequence of Theorem~\ref{thm:sol-diag-RE}.

\begin{theorem}\label{thm:sol-std-tr-RE}
	Let $\Uqk\subset\UqLg$ be an affine $\tau$-restrictable QSP subalgebra and $V, W$ small amenable finite-dimensional $\UqLg$-modules. 
	\vspace{0.25cm}
	\begin{enumerate}\itemsep0.25cm
		\item If $\tau=\eta_0$, up to a uniquely determined shift in $V$, there exists (up to a scalar multiple) a unique QSP intertwiner 
		\begin{equation} 
			\rKM{V}{z}:\shrep{V}{z}\to\shrep{V}{z^{-1}}
		\end{equation} 
		Moreover, $\rKM{V}{z}$ is a solution of the standard reflection equation \eqref{eq:rational-original-RE}.
		\item If $\tau=\id$, up to a uniquely determined shift in $V$, there exists (up to a scalar multiple) a unique QSP intertwiner
		\begin{equation} 
			\rKM{V}{z}:\shrep{V}{z}\to\shrep{V^*}{z^{-1}}
		\end{equation} 
		Moreover, under the identification of $V$ and $V^*$ as vector spaces, $\rKM{V}{z}$
		is a solution of the transposed reflection equation \eqref{eq:rational-transposed-RE}.
	\end{enumerate}
\end{theorem}

\begin{proof} \mbox{}
\begin{enumerate}\itemsep0.25cm
\item
The result follows immediately from Theorem~\ref{thm:sol-diag-RE} and equation \eqref{eq:bala-kolb-re}.
\item
By \cite[Sec.~1]{Cha02}, there exists an integer $c\in\bbZ$ depending only on $\g$ such that, for any irreducible finite-dimensional $\UqLg$-module $W$, one has $W^{\eta_0}\simeq W^*(q^c)$.
Then, proceeding as in Theorem~\ref{thm:sol-diag-RE} and relying on the latter identification, we obtain a QSP intertwiner $\rKM{V}{z}: V( z)\to V^*(z^{-1})$. 
Since in any quasitriangular Hopf algebra with universal R-matrix $R$ one has $(S\ten\id)(R)=R^{-1}$, the equation \eqref{eq:bala-kolb-re}
reduces to the transposed reflection equation \eqref{eq:rational-transposed-RE} through the identification of $V^*$ and $V$ as vector spaces. \hfill \qedhere
\end{enumerate}
\end{proof}

By Remark~\ref{rmk:shift} \ref{rmk:shift:2}, the cases $\tau=\id$ and $\tau=\eta_0$ encompass all $\tau$-restrictable QSP subalgebra (except those in type $\sfD^{(1)}_{2n}$ with non-trivial diagram automorphism).

In \cite{RV16}, solutions of the standard and transposed reflection equations on the vector representation of quantum loop algebras of classical Lie type are computed in terms of QSP intertwiners. 
By Theorem~\ref{thm:sol-std-tr-RE} we obtain the following result.

\begin{corollary}\label{cor:rv-sol}
Every explicit solution constructed in \cite{RV16} for $\tau$-restrictable QSP subalgebras (except in type $\sfD^{(1)}_{2n}$ with $\tau \ne \id$)
arises from the action of a universal K-matrix.
\end{corollary}

\subsection{Trigonometric K-matrices on Kirillov-Reshetikhin modules}\label{ss:KR-RE}

Kirillov-Reshetikhin modules, see \cite{KR87}, are minimal affinizations of irreducible $U_q(\g)$-modules 
of highest weight a multiple of a fundamental weight, see, \eg~ \cite{Her06} and references therein.
More precisely, for any $i\in I$, $k\in\bbZ_{> 0}$, and $a\in\bbC^\times$, 
the Kirillov-Reshetikhin module $W_{k,a}^{(i)}$ is the unique irreducible $\UqLg$-module whose
Drinfeld polynomials are all trivial except for the node $i$, where the roots are given by a $q_i$-string of length $k$ starting at $a$, \cf~\cite{CP94a}. For these modules, we obtain the following analogue of Theorems~\ref{thm:sol-diag-RE} and \ref{thm:sol-std-tr-RE}.
\begin{theorem}\label{thm:sol-KR-RE}
	Let $\Uqk\subset\UqLg$ be an affine $\tau$-restrictable QSP subalgebra and 
	$W$ a Kirillov-Reshetikhin $\UqLg$-module. 
	\vspace{0.25cm}
	\begin{enumerate}[leftmargin=2em]\itemsep0.25cm
		\item \label{thm:sol-KR-RE:1} 
		Up to a uniquely determined shift in $W$, there exists (up to a scalar multiple) a unique QSP intertwiner
		\begin{equation} 
			\rKM{W}{z}:\shrep{W}{z}\to\shrep{W^{\eta_0\tau}}{z^{-1}}
		\end{equation} 
		Moreover, $\rKM{W}{z}$ is a solution of the diagrammatic reflection equation \eqref{eq:bala-kolb-re}.
		\item \label{thm:sol-KR-RE:2} 
		If $\tau=\eta_0$, up to a uniquely determined shift in $W$, there exists (up to a scalar multiple) a unique QSP intertwiner
		\begin{equation} 
			\rKM{W}{z}:\shrep{W}{z}\to\shrep{W}{z^{-1}}
		\end{equation} 
		Moreover, $\rKM{W}{z}$ is a solution of the standard reflection equation \eqref{eq:rational-original-RE}.
		\item \label{thm:sol-KR-RE:3}
		If $\tau=\id$, up to a uniquely determined shift in $W$, there exists (up to a scalar multiple) a unique QSP intertwiner
		\begin{equation} 
			\rKM{W}{z}:\shrep{W}{z}\to\shrep{W^*}{z^{-1}}
		\end{equation} 
		Moreover, under the identification of $W$ and $W^*$ as vector spaces, $\rKM{W}{z}$
		is a solution of the transposed reflection equation \eqref{eq:rational-transposed-RE}.
	\end{enumerate}
\end{theorem}

\begin{proof}
Parts \ref{thm:sol-KR-RE:2} and \ref{thm:sol-KR-RE:3} follow from \ref{thm:sol-KR-RE:1} as in Theorem~\ref{thm:sol-std-tr-RE}. 
For \ref{thm:sol-KR-RE:3}, it is enough to observe that, if $W=W_{k,a}^{(i)}$, then 
\begin{equation}\label{eq:KR}
W^\omega\simeq W^{\eta_0}(a^{-2}q_i^{-2(k-1)})\,,
\end{equation}
see, \eg \cite{Cha02}.
By Theorem~\ref{thm:rational-k}, the semi-standard K-matrix 
(\cf~Remark~\ref{rmk:distinguished} \ref{rmk:distinguished:2}) provides a QSP intertwiner
$W(z)\to W^{\omega\tau}(z^{-1})$, which solves the generalized reflection equation.
We compose it with the isomorphism \eqref{eq:KR} and proceed as in Theorem~\ref{thm:sol-diag-RE}
to obtain the desired QSP intertwiner $\rKM{W}{z}$ satisfying \eqref{eq:bala-kolb-re}.
\end{proof}

\begin{remark}
The result above holds more generally for every finite-dimensional irreducible module $W$ satisfying a generalization of \eqref{eq:KR}, that is $W^\omega$ and $W^{\eta_0}$ are isomorphic up to a shift. In type $\mathsf{ADE}$, for instance, this is true if and only if there exists $c\in\bsF^\times$ such that the set of roots of each Drinfeld polynomial is invariant under the transformation $z\mapsto cz^{-1}$.
\rmkend
\end{remark}

In the case of quasi-split QSP subalgebras of affine type $\sfA$ where $\tau=\id$ (type $\sfA.1$ in the classification of \cite{RV16}), $\tau$ is the half-turn rotation of the Dynkin diagram (type $\sfA.4$) or $\tau=\eta_0$ (types $\sfA.3a$ and $\sfA.3b$), a combinatorial formula for the trigonometric K-matrices given in parts \ref{thm:sol-KR-RE:2} and \ref{thm:sol-KR-RE:3} above has been obtained in \cite{KOW24}. 
Theorem \ref{thm:sol-KR-RE} now implies the following result.

\begin{corollary}\label{cor:kow-sol}
Every explicit solution in \cite{KOW24} for QSP subalgebras of type $\sfA.1$ and $\sfA.3$ arises from the action of a universal K-matrix.
\end{corollary}

\subsection{The q-Onsager algebra and the vector representation }\label{ex:q-Ons-K}
We conclude with the explicit formula of the trigonometric K-matrix $\rKM{}{z}$ arising from the q-Onsager algebra $\Uqk$ on the vector representation $V$ of $U_qL{\mathfrak{sl}}_2$, which first appeared in \cite{DVGR94}. 
In particular, we observe that the poles of $\rKM{}{z}$ completely detect the irreducibility of $V$ under restriction to $\Uqk$, only for a {generic} choice of the QSP parameters (\cf~Propositions~\ref{prop:unitarity}, \ref{prop:unitarity:case2}). 
More specifically, the denominator in $\rKM{}{z}$ on $V$ is \emph{generically} a polynomial of degree 2, but for QSP parameters satisfying specific closed conditions it reduces to a linear and even constant polynomial. 
In the latter cases, one checks that there are values $\zeta\in\bsF^\times$ such that $V(\zeta)$ is reducible \emph{and} the operator $\rKM{}{\zeta}$ is well-defined.\\

Let $V\coloneqq V_{\omega_1}(q)$ be the 2-dimensional 
representation with action given by 
\begin{gather}
	\pi(E_0) = \begin{pmatrix} 0 & 0 \\ 1 & 0 \end{pmatrix} = \pi(F_1)\,,
	\qq
	\pi(F_0) = \begin{pmatrix} 0 & 1 \\ 0 & 0 \end{pmatrix}= \pi(E_1)\,,\\
	\pi(K_0) = \begin{pmatrix} q^{-1} & 0 \\ 0 & q \end{pmatrix} =\pi(K_1)^{-1}\,,
\end{gather}
see~Remark~\ref{rmk:spectral-R-matrix}.
Let $\Uqk\subset U_q(\wh{\mathfrak{sl}}_2)$ be the q-Onsager algebra from Example \ref{exam:qOnsager} with QSP parameters
$\parc{0},\parc{1}\in\bsF^\times$, $\pars{0},\pars{1}\in\bsF$.
Since $V$ is amenable and $\tau=\id=\eta_0$, 
Theorem~\ref{thm:sol-diag-RE} applies and yields
a rational QSP intertwiner 
\[\rKM{}{z}:\shrep{V}{\kappa z}\to\shrep{V}{\kappa/z}\] 
where $\kappa^{-2}={\parc{0}\parc{1}}$ (see also~\cite[Cor.~9.2]{AV22}).
This satisfies equation \eqref{eq:bala-kolb-re},
which in this particular case reduces to the standard reflection equation \eqref{eq:rational-original-RE}.\\

Up to a scalar, $\rKM{}{z}$ can be obtained by computing
directly a polynomial non--vanishing solution of the intertwining equation \eqref{QSPintertwiner:shifted},
relying on the explicit formulas of the QSP generators \eqref{qOnsager:generators}, and by further imposing the
unitarity condition $\rKM{}{z}^{-1} = \rKM{}{z^{-1}}$ according to Proposition~\ref{prop:unitarity:case2}.
For convenience, we replace the QSP parameters by choosing $\la,\mu,\nu \in \bsF^\times$ such that 
\[
\parc{0} = \kappa^{-2} \nu^2, \qq \parc{1} = \nu^{-2},\qq \pars{0} = \kappa^{-1} \nu \frac{\mu + \mu^{-1}}{q-q^{-1}}, \qq \pars{1} = - \nu^{-1} \frac{\la + \la^{-1}}{q-q^{-1}}. 
\]
\begin{enumerate}[leftmargin=2em]\itemsep0.25cm
	\item
	If $\lambda\mu^{\pm1}\neq\pm1$, we get
	\[
	\rKM{}{z} = \frac{\lambda\mu}{(\lambda \mu + z)(\lambda+\mu z)}\;
	\begin{pmatrix} & \left(\mu + \frac{1}{\mu}\right) z + \left(\lambda+\frac{1}{\lambda}\right) z^2 & -\frac{1}{\nu} \left(z^2-1\right) &\\ &&&\\ &\nu \left(z^2-1\right) & \left(\mu + \frac{1}{\mu}\right) z  + \lambda+\frac{1}{\lambda} &\end{pmatrix}
	\vspace{0.25cm}
	\]
	In particular, $\rKM{}{\pm1}=\id$.
	
	\item
	If $\lambda\neq\pm i$ and $\lambda\mu^{\pm1}=\varepsilon$ with $\varepsilon\in\{\pm1\}$, then \\[0mm]
	\[
	\rKM{}{z} =\frac{1}{\varepsilon\lambda+\frac{1}{\lambda}z} \;
	\begin{pmatrix} & \left(\lambda + \frac{1}{\lambda}\right) z & -\frac{1}{\nu} \left(z-\varepsilon\right) &\\ &&&\\ &\nu \left(z-\varepsilon\right) & \varepsilon(\lambda+\frac{1}{\lambda}) &\end{pmatrix}
	\vspace{0.25cm}
	\]
	In particular, $\rKM{}{\varepsilon}=\id$. 
	Note however that $V(-\varepsilon \kappa)$ is reducible. 
	If $\lambda=\pm1$, $\rKM{}{z}$ has a pole at $z=-\varepsilon$, which therefore detects the reducibility of $V(-\varepsilon \kappa)$. 
	On the other hand, if $\lambda\neq\pm1$, $\rKM{}{-\varepsilon}$ is still well--defined.
	
	\item 
	Finally, if $\lambda=\pm i$ and $\lambda\mu^{\pm 1}=\varepsilon$ with $\varepsilon\in\{\pm1\}$, then \\[0mm]
	\[
	\rKM{}{z} =
	\begin{pmatrix} & 0 & -\frac{1}{\nu}&\\ &&&\\ &\nu  & 0 &\end{pmatrix}
	\vspace{0.25cm}
	\]
	In this case $\pars{0}=0=\pars{1}$ and $V(\pm \kappa)$ is reducible. 
	However, $\rKM{}{z}$ is constant and always well-defined.
\end{enumerate}


\providecommand{\bysame}{\leavevmode\hbox to3em{\hrulefill}\thinspace}
\providecommand{\MR}{\relax\ifhmode\unskip\space\fi MR }
\providecommand{\MRhref}[2]{%
	\href{http://www.ams.org/mathscinet-getitem?mr=#1}{#2}
}
\providecommand{\href}[2]{#2}


\end{document}